\documentclass[aos,preprint]{imsart}

\RequirePackage[OT1]{fontenc}
\RequirePackage{amsthm,amsmath}
\RequirePackage[numbers]{natbib}
\RequirePackage[colorlinks,citecolor=blue,urlcolor=blue]{hyperref}

\usepackage{graphicx} 
\usepackage{mathrsfs}
\usepackage{amssymb}
\usepackage{amsmath,bm}
\usepackage{mathtools}
\usepackage{frame,array}
\usepackage{indentfirst} 
\usepackage{amsfonts}
\usepackage{tabularx}

\usepackage{tikz}
\usetikzlibrary{plotmarks}
\usetikzlibrary{arrows,snakes,backgrounds}
\usetikzlibrary{calc}
\usepackage{relsize}

\numberwithin{equation}{section}
\theoremstyle{plain}
\newtheorem{theorem}{Theorem}[section]
\newtheorem*{namedtheorem}{\theoremname}
\newcommand{\theoremname}{testing}

\newtheorem{claim}[theorem]{Claim}

\theoremstyle{definition}
\newtheorem{definition}[theorem]{Definition}
\newtheorem{remark}[theorem]{Remark}

\usepackage{amssymb,amsmath,amsthm}
\usepackage{epsf}
\usepackage{times,bm,mathrsfs}
\usepackage{amsmath}
\usepackage{amssymb}
\usepackage{amsfonts}
\usepackage{mathrsfs}
\usepackage{bbm}
\usepackage{color}
\usepackage{graphicx}
\usepackage{amscd}
\usepackage{epsfig}
\usepackage{comment}
\usepackage{tikz, tikz-cd}
\usetikzlibrary{matrix}
\usetikzlibrary{fit, positioning, arrows}

\definecolor{Red}{rgb}{1,0,0}
\definecolor{Blue}{rgb}{0,0,1}
\definecolor{Olive}{rgb}{0.41,0.55,0.13}
\definecolor{Green}{rgb}{0,1,0}
\definecolor{MGreen}{rgb}{0,0.8,0}
\definecolor{DGreen}{rgb}{0,0.55,0}
\definecolor{Yellow}{rgb}{1,1,0}
\definecolor{Cyan}{rgb}{0,1,1}
\definecolor{Magenta}{rgb}{1,0,1}
\definecolor{Orange}{rgb}{1,.5,0}
\definecolor{Violet}{rgb}{.5,0,.5}
\definecolor{Purple}{rgb}{.75,0,.25}
\definecolor{Brown}{rgb}{.75,.5,.25}
\definecolor{Grey}{rgb}{.5,.5,.5}
\definecolor{Black}{rgb}{0,0,0}

\newcommand{\bfpi}{\boldsymbol{\pi}}

\newcommand{\eps}{\varepsilon}
\newcommand{\ind}{\mathbf{1}}
\newcommand{\E}{\mathbb{E}}
\renewcommand{\P}{\mathbb{P}}
\newcommand{\var}{\mathrm{Var}}
\newcommand{\cov}{\mathrm{Cov}}

\newcommand{\gr}{G}
\newcommand{\ve}{V}
\newcommand{\ed}{E}
\newcommand{\tree}{\mathbbm{T}}

\newcommand{\nint}{n_{\mathrm{in}}}
\newcommand{\dmax}{d_{\mathrm{max}}}
\newcommand{\nve}{N}
\newcommand{\block}[1]{z(#1)}

\newcommand{\bsize}[1]{\nve_{#1}}

\newcommand{\aff}[1]{B_{#1}}
\newcommand{\btr}[2]{p_{#2}(#1)}
\newcommand{\pbtr}[1]{p_{#1}}
\newcommand{\epbtr}[1]{\hat{p}_{#1}}
\newcommand{\pbtrmat}{P}

\newcommand{\func}{y}
\newcommand{\bfunc}{c_y}
\newcommand{\mfunc}{\mu}
\newcommand{\wfunc}[1]{\hat{\mfunc}_{{#1},\mathrm{w}}}
\newcommand{\emfunc}[1]{\hat{\mfunc}_{{#1}\mathrm{VH}}}
\newcommand{\ps}{\hat{\mfunc}_{\mathrm{PS}}}
\newcommand{\ipw}{\hat{\mfunc}_{\mathrm{IPW}}}
\newcommand{\vh}{\hat{\mfunc}_{\mathrm{VH}}}
\newcommand{\mvh}{\hat{\mfunc}_{\mathrm{mVH}}}
\renewcommand{\deg}[1]{d_{#1}}

\newcommand{\edeg}[1]{\hat{H}_{#1}}

\newcommand{\mdegb}[1]{\delta^B_{#1}}
\newcommand{\degtwow}{d_\theta^{(2)}}
\newcommand{\degtwo}{d^{(2)}}
\newcommand{\neighb}{\mathcal{N}}

\newcommand{\pstat}[1]{\pi_{#1}^{B}}
\newcommand{\estat}[1]{\hat{\pi}_{#1}^{B}}
\newcommand{\pstatvec}{\bfpi^{B}}
\newcommand{\estatvec}{\hat{\bfpi}^{B}}
\newcommand{\filter}{\mathcal{F}}

\newcommand{\filtertwo}{\mathcal{G}}
\newcommand{\childtwo}{\mathcal{C}^{(2)}}
\newcommand{\etrans}{n}
\newcommand{\eventdeg}{\mathcal{E}_{\mathrm{D}}}
\newcommand{\eventdegtwo}{\mathcal{E}_{\mathrm{D},2}}
\newcommand{\neven}{n^{(\mathrm{e})}}
\newcommand{\etranseven}{\etrans^{(\mathrm{e})}}
\newcommand{\azuid}[1]{I_{#1}}
\newcommand{\eazu}[1]{W_{#1}}

\usepackage{graphicx,enumitem, float}
\usepackage{fancybox}
\usepackage{tikz}
\usepackage{wrapfig}
\usepackage{subfigure}

\startlocaldefs
\numberwithin{equation}{section}
\theoremstyle{plain}

\endlocaldefs

\begin{document}

\begin{frontmatter}
	\title{Reducing Seed Bias in Respondent-Driven Sampling\\by Estimating Block Transition Probabilities}
	\runtitle{Post-Stratified RDS}
	
	\begin{aug}
		\author{\fnms{Yilin} \snm{Zhang}\thanksref{t1}\ead[label=e1]{yilin.zhang@wisc.edu}},
		\author{\fnms{Karl} \snm{Rohe}\thanksref{t1}\ead[label=e2]{karlrohe@stat.wisc.edu}},
		\and
		\author{\fnms{Sebastien} \snm{Roch}\thanksref{t2}\ead[label=e3]{roch@math.wisc.edu}}

		\thankstext{t1}{These authors gratefully acknowledge support from NSF grant DMS-1612456 and ARO grant W911NF-15-1-0423.}
		\thankstext{t2}{This author gratefully acknowledges support from NSF grants DMS-1149312 (CAREER), DMS-1614242 and CCF-1740707 (TRIPODS), and a Simons Fellowship.}
		\runauthor{Zhang et al.}
		\affiliation{University of Wisconsin-Madison,\\ Department of Statistics and Department of Mathematics}
		
		\address{
			Yilin Zhang, Karl Rohe\\
			Department of Statistics\\
			University of Wisconsin Madison\\
			1300 University Ave\\
			Madison, WI 53706\\
			USA\\
			\printead{e1}\\
			\phantom{E-mail:\ }\printead*{e2}}
		
		\address{
			Sebastien Roch\\
		    Department of Mathematics\\
		    University of Wisconsin-Madison\\
		    480 Lincoln Drive\\
		    Madison, WI 53706\\
		    USA\\
			\printead{e3}}
	\end{aug}
\begin{abstract}
Respondent-driven sampling (RDS) is a popular approach to study marg\-inalized or hard-to-reach populations.  It collects samples from a networked population by incentivizing participants to refer their friends
into the study.  One major challenge in analyzing RDS samples is seed bias.  Seed bias refers to the fact that when the social network is divided into multiple communities (or blocks), the RDS sample might not provide a balanced representation of the different communities in the population, and such unbalance is correlated with the initial participant (or the seed).  In this case, the distributions of estimators are typically non-trivial mixtures, which are determined (1) by the seed and (2) by how the referrals transition from one block to another.  This paper shows that 
(1) block-transition probabilities are easy to estimate with high accuracy, and 
(2) we can use these estimated block-transition probabilities to estimate the stationary distribution over blocks and thus, an estimate of the block proportions.  This stationary distribution on blocks has previously been used in the RDS literature to evaluate whether the sampling process has appeared to ``mix''.  
We use these estimated block proportions in a simple post-stratified (PS) estimator that greatly diminishes seed bias.
By aggregating over the blocks/strata in this way, we prove that the PS estimator is $\sqrt{n}$-consistent under a Markov model, even when other estimators are not.  Simulations show that the PS estimator has smaller Root Mean Square Error (RMSE) compared to the state-of-the-art estimators.

\end{abstract}


\begin{keyword}
	\kwd{respondent-driven sampling}
	\kwd{post-stratification}
	\kwd{social network}
	\kwd{Stochastic Blockmodel}
	\kwd{Markov process}
\end{keyword}

\end{frontmatter}

\section{Introduction}

Respondent-driven sampling (RDS) is one of the most popular network-based approaches to sample marginalized and hard-to-reach populations, such as drug users, sex workers, and the homeless~\cite{heckathorn1997respondent}.  RDS has been widely used, for instance, to quantify HIV prevalence in at-risk populations \cite{malekinejad2008using, johnston2013introduction}.   According to a recent literature review \cite{white2015strengthening}, RDS has been used in over 460 studies from 69 countries.

\begin{figure}\centering 
	\includegraphics[width=1\columnwidth]{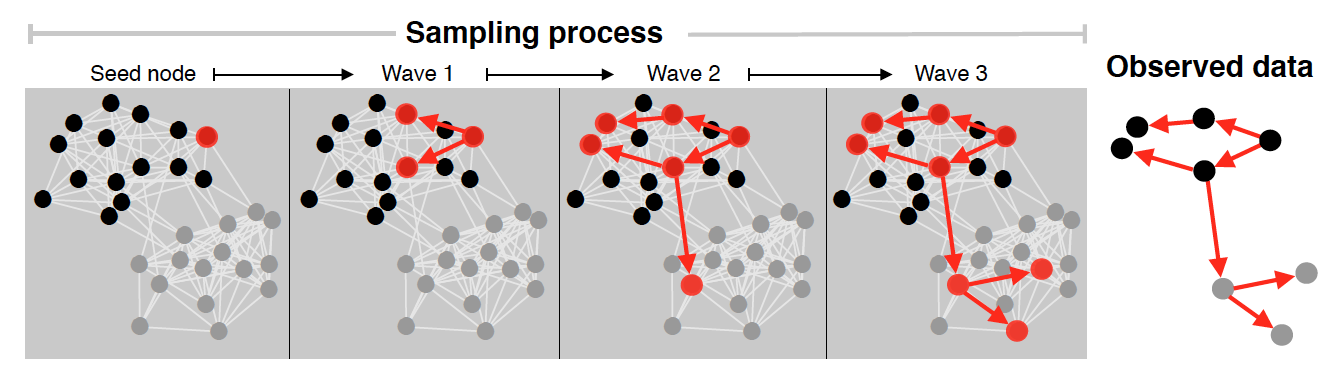}
	\caption{This Figure from \cite{rohe2015network} illustrates the RDS sampling process.  }
	\label{fig: rds_process}
\end{figure}

RDS collects samples through peer referral on a social network.  It starts from some initial participant as the seed, which forms wave zero.  In the process, we incentivize each participant to pass some (usually three to five) referral coupons to their friends. Those who return to the study site with a referral coupon form the next wave of samples.  We repeat this process until we get enough samples or the participants stop referring.  Figure \ref{fig: rds_process} from \cite{rohe2015network} gives an illustration for the RDS sampling process.  There are three components in RDS sampling: (1) the social network, (2) the sampling tree, and (3) the variable of interest (denoted by color in Figure \ref{fig: rds_process}).  The underlying social network is the target population to study, which is unobserved.  For each sampled node, we observe their HIV status (black or grey in Figure \ref{fig: rds_process}),  and which node refers them to the sample.  We aim to estimate the proportion of people with certain trait, such as HIV positive (nodes that are grey in Figure \ref{fig: rds_process}), in the population.  

The link-tracing sampling procedure of RDS enables us to reach the hard-to-reach populations.  However, 
RDS samples are dependent.  This dependence is particularly bad when there are multiple communities in the target population and the people form most of their friendships within their own communities (i.e. blocks).  For example, people from the east side of the town might only know a few people from the west side of the town, and thus they are much more likely to refer people from the west side of the town.  
This is referred to as a ``bottleneck'' and it leads to a sample that is unbalanced between the different communities. 
If the HIV prevalence is higher on one side of the town, then this bottleneck creates dependence between observations in an RDS sample.  
If the initial participant is from the east side, then the sample may underrepresent people from the west side.  This creates ``seed bias.''  
In statistical models which presume that the seed node is randomized, this ``seed bias'' appears as additional variance in the final estimator.  When some participants refer too many contacts, the variance of the traditional RDS estimator, Volz-Heckathorn (VH) estimator \cite{heckathorn1997respondent}, decays at a rate slower than $O(n^{-1})$ \cite{rohe2015network}.  We provide an example in Appendix \ref{app:negative}.   To address this issue, recent work \cite{roch2017generalized} has derived an idealized generalized least squares (GLS) estimator for which the standard error decays at rate $O(n^{-1/2})$ with growing sample size $n$ under a fixed social network.  The practical implementation of the estimator, called the feasible GLS (fGLS) estimator, requires solving an $n \times n$ system of equations and comes with no theoretical guarantees.

\begin{figure}[t] 
   \centering
   \includegraphics[width=5in]{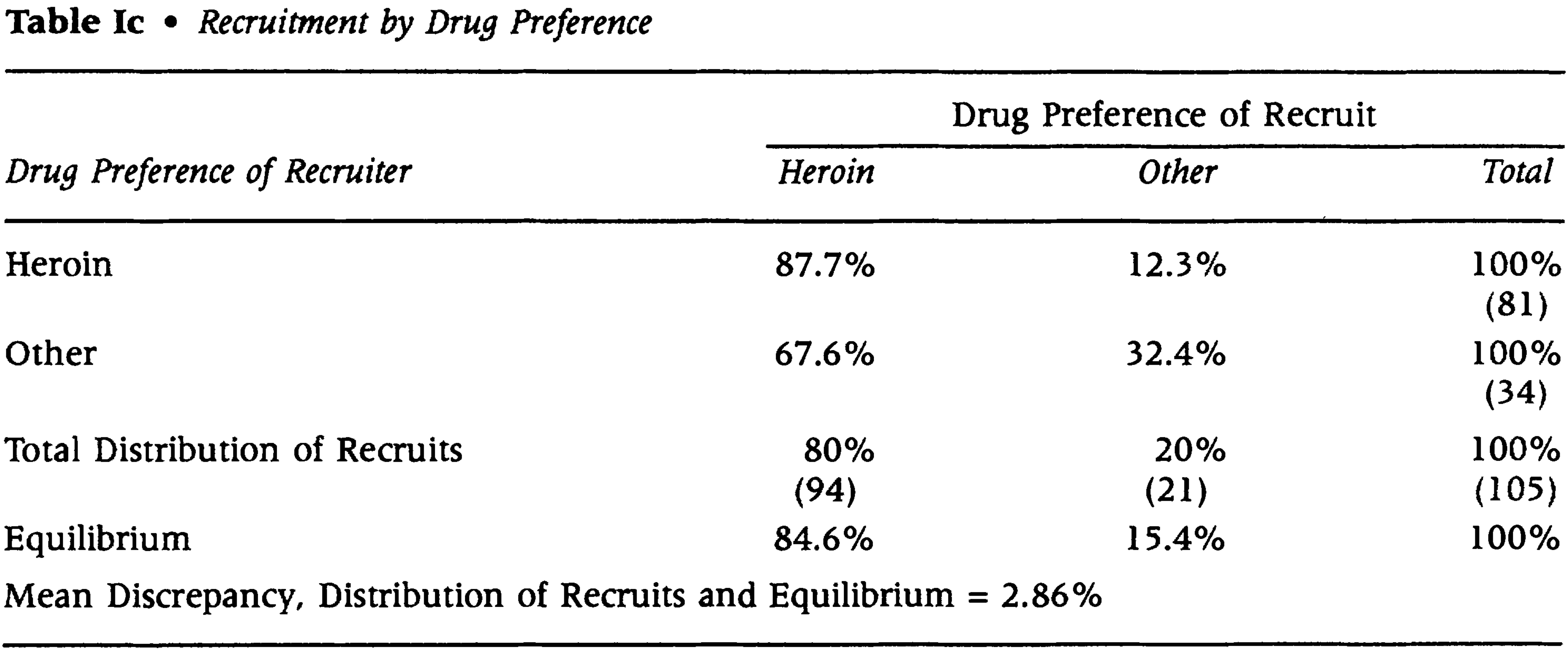} 
   \caption{
Heckathorn (1997) first proposed RDS and illustrated the RDS technique with a sample of drug users.  This is Table 1c from that paper.  It summarizes the sample by computing the empirical transition matrix between two strata of drug users; those who prefer heroin and those who prefer some other drug.  Other empirical transition matrices in that paper stratify based upon ethnicity, gender, and location of recruitment.}
   \label{fig:heck}
\end{figure}

This paper provides an estimator that is easy to compute and has root mean squared error that decays at rate $\Theta(n^{-1/2})$ up to log factors, by implicitly adjusting for bottlenecks between different communities.  While this estimator is new, its essential components are well known and reported in the RDS literature.
This new estimator assumes that we have collected the ``bottlenecked'' community memberships of the sampled individuals.  
With this data, a key summary is the empirical transition matrix between communities, in which element $u,v$ is the proportion of referrals from participants in community $u$ to participants in community $v$. 
In the RDS literature, this matrix is a common way to summarize the sampling procedure and understand the underlying social network.  
For example, the original RDS paper \cite{heckathorn1997respondent} reports on a sample of drug users.  Table 1c from that paper (reprinted as Figure \ref{fig:heck} herein) gives the empirical transition matrix between communities defined by drug preference.  
This empirical transition matrix is also a key piece of the feasible GLS estimator \cite{roch2017generalized}.

Interestingly, an estimate of the proportion of nodes in each community can be derived from the empirical transition matrix.
Notice in Figure \ref{fig:heck} that \cite{heckathorn1997respondent} reports the equilibrium distribution on the different strata/communities. This takes the empirical transition matrix  as a Markov transition matrix on the different communities and computes the stationary (i.e. equilibrium) distribution of this Markov process (i.e. the leading left eigenvector of the transition matrix).  In Figure \ref{fig:heck}, the equilibrium distribution is close to the total distribution of recruits.  When there is a bottleneck, this paper shows that the equilibrium distribution is a better estimator than the total distribution of recruits. The basic reason is that \textit{even when there is a bottleneck}, each row of the empirical transition matrix is composed of $O(n)$ nearly independent multinomial samples.  There is one caveat; our estimator does not use the actual equilibrium distribution of the empirical transition matrix (i.e. the quantity reported in Figure \ref{fig:heck}).  Instead, we have a simple approximation of the equilibrium which is easier to compute and thus simplifies the proof.  

The final estimator is a post-stratified estimator where the strata are the community memberships and the estimated proportion of nodes in each strata is derived from the estimated equilibrium distribution. We call this the PS estimator.
%
%
%
%
%
%
The PS estimator has three major advantages: (1) computational efficiency, (2) smaller variation (bias square, variance and RMSE), and (3) block-wise byproducts.  We show in Theorem \ref{thm: consistency} that our PS estimator has both its bias and standard deviation decay at rate $\Theta(n^{-1/2})$ up to log factors, 
which does not hold for the popular Volz-Heckathorn (VH) esimtator \citep{heckathorn1997respondent} and does not show the GLS estimator \citep{roch2017generalized}.  The simulation studies also show our PS estimator has smaller variation (bias square, variance and RMSE) compared to the VH estimator and fGLS estimator.  The improvement is significant especially when there exists bottleneck in social networks.


The paper is organized as follows.  Section \ref{sec: preliminary} defines the Markov model, the quantity to estimate, and the traditional RDS estimators.  Section \ref{sec: stratified_vh} introduces the PS estimator.  Section \ref{sec: theory} shows PS estimator is $\sqrt{n}$-consistent under the Degree Corrected Stochastic Blockmodel (DC-SBM).  In Section \ref{sec: simulation}, we show by simulations that PS estimator has smaller variation than the state-of-the-art estimators, especially when there exists bottleneck in social networks.  We summarize with a discussion in Section \ref{sec: discuss}.
\section{Preliminaries}\label{sec: preliminary}

We model referrals using a Markov process similar to the ones
previously considered in the RDS literature \cite{goel2009respondent, heckathorn1997respondent, salganik2004sampling, volz2008probability,rohe2015network,roch2017generalized}. 

\subsection{Markov process on a social network}\label{sec: markov}

A social network $G$ consists of a node set 
$V = \{1,\dots, N\}$ of individuals
and an undirected edge set 
$$E = \{\{i,j\}: \text{$i$ and $j$ can refer one another}\}.$$  
We use $i\in V$ and $i\in G$ interchangeably.  We assume that $G$ is connected.
Let $w_{ij} = w_{ji} > 0$ be the weight of edge $\{i,j\}\in E$, which models recruitment preference (more details in Section \ref{sec: theory}).  For any $\{i,j\}\not\in E$, we let $w_{ij} = w_{ji} = 0$ by convention.  
If the graph is unweighted, then $w_{ij} = 1$ for all $\{i,j\}\in E$. 
For each node $i \in V$, we denote its neighbor in the network $G$ by
$
\neighb(i)
=
\left\{
j \in V\,:\,\{i,j\} \in E
\right\}.$  
We denote the degree of node $i$ as $\deg{i} = \sum_j w_{ij}$ and the mean degree of graph $G$ as $\bar{d}= \sum_i \deg{i}/N$.

We model the collection of samples in RDS with a Markov process on the social network $G$ indexed by a tree.  It starts with an initial participant as seed, which we index as vertex 0, and develops into a rooted tree, $\tree$ (a connected graph with $n$ nodes, no cycles, and a vertex $0$). We use $\tau\in\tree$ to denote that node $\tau$ belongs to the samples indexed by $\tree$.  
For each node $\tau\in\tree$, we denote the parent of $\tau$ as $\tau'$ (the node that refers $\tau$ to the sample).
Formally, an RDS sample is an indexed collection of random nodes $(X_{\tau}\in G: \tau\in\tree)$, where each referral $X_{\tau'}\rightarrow X_{\tau}$ 
has probability 
$$\mathbbm{P}(X_{\tau} = j|X_{\tau'} = i) = P_{ij}, \quad \forall i,j\in G,$$
where the transition matrix $P\in\mathbbm{R}^{N\times N}$ has elements $$P_{ij} = \frac{w_{ij}}{\deg{i}}.$$
Since the graph $G$ is undirected and connected, $P$ is a reversible Markov transition matrix with unique stationary distribution $\bfpi = (\pi_i)_{i \in G} \in \mathbbm{R}^N$ with $$\pi_i = \frac{\deg{i}}{N\bar{d}}.$$  
While the referrals are random, we think of $\tree$ itself as deterministic.

Following \cite{benjamini1994markov}, we refer to this Markov process as a $(\tree, P)$-walk on $G$.  Note that $G$ and $\tree$ are two distinct graphs: the node set in $G$ indexes the population, which is a social network, and the node set in $\mathbb{T}$ indexes the samples, which is a sampling tree.  
We say that the $(\tree, P)$-walk is stationary if the seed
is chosen according to the stationary distribution.

\subsection{Quantity to estimate and the Volz-Heckathorn estimator}
For each node $i\in G$, we denote the variable of interest (e.g., the indicator of HIV status) as $y(i)$.  We wish to estimate the population mean of the variable of interest
$$\mu_{\text{true}} = \frac{1}{N}\sum\limits_{i\in G} y(i).$$
For each sample $X_{\tau}$, we observe 
$$Y_{\tau} = y(X_{\tau}), \quad\forall\tau\in \tree.$$  
The sample average 
$$\hat{\mu} = \frac{1}{n}\sum\limits_{\tau\in \tree}Y_{\tau}$$ 
is generally biased, since nodes with larger degrees are more likely to be sampled in the Markov process. 
Specifically, under the \emph{stationary} $(\tree,P)$-walk on $G$, it has expectation 
$$\mathbbm{E}[\hat{\mu}] = \mu = \mathbbm{E}[Y_0] = \sum\limits_{i \in G} y(i)\pi_i.$$
In general, $\mu\not=\mu_{\text{true}}$.  

To obtain an unbiased estimator of $\mu_{\text{true}}$, the sample average must be adjusted.  Using $\pi_i = \deg{i}/(N\bar{d})$, the inverse probability weighted estimator (IPW), 
$$\ipw = \frac{1}{n}\sum\limits_{\tau\in \tree} \frac{Y_{\tau}}{\pi_{X_{\tau}}N} = \frac{\bar{d}}{n}\sum\limits_{\tau\in\tree}\frac{Y_{\tau}}{\deg{X_{\tau}}},$$
is an unbiased estimator of $\mu_{\text{true}}$ \citep{horvitz1952generalization}.  Additionally estimating 
$\bar{d}$ with the harmonic mean of the observed node degrees,
$$\hat{H} = \left(\frac{1}{n}\sum\limits_{\tau\in\tree}
\frac{1}{\deg{X_{\tau}}}\right)^{-1},$$
leads to the popular Volz-Heckathorn (VH) estimator \citep{volz2008probability},
$$\vh = \frac{\hat{H}}{n}\sum\limits_{\tau\in\tree}\frac{Y_{\tau}}{\deg{X_{\tau}}}.$$

The VH estimator has been extensively used in the study of marginalized populations \cite{malekinejad2008using, johnston2013introduction, white2015strengthening}, but it is highly variable.  The variance of the VH estimator in general may decay at a rate slower than $O(n^{-1})$ \cite{rohe2015network}, implying that many more samples are required to reduce the standard error. See Section~\ref{app:negative}. We address this issue by introducing a post-stratification approach to RDS in the following section. 

\section{A new estimator}
\label{sec: stratified_vh}

\subsection{A post-stratification approach to RDS} \label{sec: stratification}


\paragraph{Stratification}
Stratification has been extensively used in traditional random sampling to reduce variance.  The key idea of stratified sampling is as follows. Assume that the overall population can be divided into (ideally homogeneous) sub-groups (which we refer to as blocks) based on some variable, such as gender, race, etc.  Then the sample mean and sample variance of the total population can be calculated using block-wise sample means and variances.

Specifically, suppose there are $K$ blocks in a population with $N$ individuals.  For each block $k$, we denote the block size as $N_k$, the block-wise population mean as $\mu_k$, the sample size as $n_k$
and the block-wise sample average as $\hat{\mu}_k$.  The sample average $\hat{\mu}$ and sample variance $s^2$ for the total population can be derived from the block-wise quantities by
\begin{equation}\label{eq:stratified-formulas}
\hat{\mu} = \sum\limits_{k = 1}^K \left(\frac{N_k}{N}\right) \hat{\mu}_k,
\quad\text{ and }\quad s^2 = \sum\limits_{k = 1}^K \left(\frac{N_k}{N}\right)^2\frac{N_k-n_k}{N_k}\frac{s_k^2}{n_k}.
\end{equation}

Stratified sampling by proportionate allocation randomly selects individuals proportionally to the sizes of the different blocks, with the
goal of improving accuracy by reducing sampling error. Post-stratified sampling, on the other hand, performs stratification \emph{after sampling} and calculates $\hat{\mu}$ and $s^2$ as above.  Post-stratification is useful when the samples constitute an unbalanced representation of the full population.

\paragraph{Block proportions are unobserved in marginalized populations}
We seek to apply this last approach
to RDS in order to deal with seed bias. An important issue arises
however. Per~\eqref{eq:stratified-formulas},
traditional post-stratification requires the knowledge of the block proportions $N_k/N$.  These are typically unknown in marginalized populations.  Hence, we need to estimate the block proportions from the samples.  In the next section, we describe how we do this and 
we formally define a novel post-stratified estimator for RDS.

\subsection{Block-wise quantities}

For a set $V'$, denote its cardinality by $|V'|$.   
Suppose there are $K$ blocks 
in the social network $G$. For each node $i\in G$, denote its block membership as $z(i)$, i.e., $z(i) = k$ if $i$ belongs to block $k\in\{1,\dots, K\}$.  To simplify notation, we write $i\in V_k$ to mean $z(i) = k$.  For each block $k$, we denote the block size as $N_k = |V_k|$ 
and the block-wise mean as $\mu_k = N_k^{-1}\sum\limits_{i\in V_k}y(i)$.  

For each sample $\tau\in\tree$, we let its block membership be $Z_{\tau} = z(X_{\tau})$ and we write $\tau\in \tree_k$ to mean $Z_{\tau} = k$.  We define for each block $k$ the sample size as $n_k$,  the block-wise harmonic average degree as 
\begin{equation}\label{def: H_delta_k}
\edeg{k} = \left(\frac{1}{n_k}\sum\limits_{\tau\in\tree_k}\frac{1}{\deg{X_{\tau}}}\right)^{-1}, 
\end{equation}
and the block-wise sample average weighted by degree, i.e., the VH estimator for $\mu_k$, as
\begin{equation}\label{def: hat_mu_k}
\emfunc{k} = \frac{\edeg{k}}{n_k}\sum\limits_{\tau\in\tree_k}\frac{Y_{\tau}}{\deg{X_{\tau}}}.
\end{equation}

Suppose that we observe the block membership of each sample, i.e., we observe $Z_{\tau} = z(X_{\tau})$ for all $\tau\in\tree$.  We define the matrix $\hat{Q}\in\mathbbm{R}^{K\times K}$ such that, for any two blocks $u,v\in\{1,\dots,K\}$,
\begin{align*}
\hat{Q}_{uv} = \frac{1}{n}\times\text{number of referrals from block $u$ to block $v$},
\end{align*}
and the row-normalized matrix $\hat{P}^B\in\mathbbm{R}^{K\times K}$ whose $(u,v)$-entry is
\begin{equation}\label{def: hat_QR}
\epbtr{uv} = \frac{\hat{Q}_{uv}}{\hat{Q}_{u\ast}} .
\end{equation}
Here, for a matrix $A$, we let $A_{u\ast} = \sum\limits_{v} A_{uv}$
and $\mathbf{1}\{\mathcal{E}\}$ is the indicator of event $\mathcal{E}$.  
Finally we define the vector $\estatvec = (\estat{k})_k$ with entries
\begin{equation}\label{def: estat}
\estat{k}
=  \left[\sum\limits_v\frac{\epbtr{kv}}{\epbtr{vk}}\right]^{-1}.
\end{equation}

\subsection{The post-stratified estimator}

We define our new estimator next.
\begin{definition}[The post-stratified estimator]\label{def: s-vh}  For an RDS sample on a graph $G$ with $K$ blocks, the post-stratified (PS) estimator is
\begin{equation}\label{eq: s-vh}
\ps = 
\sum_k \hat{\alpha}_k\,
	\emfunc{k},
\end{equation}
with
$$
\hat{\alpha}_k = \frac{\estat{k}/\edeg{k}}{\sum_\ell \estat{\ell}/\edeg{\ell}},
$$
where $\edeg{k}$, $\emfunc{k}$, and $\estat{k}$  are defined in \eqref{def: H_delta_k}, \eqref{def: hat_mu_k} and \eqref{def: hat_QR}
respectively.
\end{definition}
Comparing~\eqref{eq: s-vh} with~\eqref{eq:stratified-formulas}, the estimator $\ps$ can indeed be seen as a post-stratified estimator. 
In Section \ref{sec: motivation}, we argue that $\hat{\alpha}_k$ is an estimator of the block proportion of block $k$.  
Note that we also use the VH estimator $\emfunc{k}$ on each block $k$, instead of the block-wise sample average, to adjust for the bias induced by node degrees.  



\subsection{Motivation for the PS estimator}\label{sec: motivation}


To motivate our new estimator, we analyze its behavior under
a standard model of random social network with community structure, the degree-corrected stochastic blockmodel (DC-SBM) \cite{karrer2011stochastic}.
\begin{definition}[Degree-corrected stochastic blockmodel] 
	Let $B\in \mathbbm{R}_{+}^{K\times K}$ be a positive, symmetric matrix and let $\theta \in \mathbbm{R}_+^N$ be a positive vector. Under the DC-SBM, a social network $G=(V,E)$ with $V = \{1,\ldots,N\}$ is drawn randomly as follows. Assume that we have a partition $V_1,\ldots,V_K$ of $V$ into $K$ blocks labeled $\{1,\ldots,K\}$. Let $N_1, \ldots, N_K$ be the respective sizes of the blocks. For a node $i \in V$, let $Z_i$ be its block. 
	Each possible edge $\{i,j\}$ is present independently from all other edges with probability 
	\begin{equation}\label{def: dc-sbm}
	\P[\{i,j\} \in E] = \theta_i \theta_j B_{Z_i, Z_j}.
	\end{equation}
	By convention, we assume $\sum_{i\in V_k}\theta_i = 1$ 
	for all block $k$.  
\end{definition}
\begin{remark}[Self-loops]
	To simplify the notation throughout, we allow self-loops
	$\{i,i\}$ in the DC-SBM, each of which will contribute $1$ to degree counts (instead of the standard convention of $2$). 
	Note that, in a dense
	graph, such self-loops will play a negligible role. 
\end{remark}

To justify our PS estimator under the DC-SBM, 
we make three observations:
\begin{enumerate}
	\item 
	Define the matrices $Q = B/m$, where $m = \ind^T B \ind$,
	and $\pbtrmat^{B}= (\pbtr{uv})_{u,v}$, where
	\begin{equation}
	\label{eq:pbtr-def}
	\pbtr{uv} 
	= \frac{\aff{uv}}{\aff{u\ast}}
	= 
	\frac{Q_{uv}}{Q_{u\ast}},
	\end{equation}
	for any two blocks $u,v$.
	Since $P^B$ is positive and row-normalized version of the symmetric matrix $Q$, it has a unique stationary distribution $\pstatvec = (\pstat{k})_k$, where 
	$$
	\pstat{k} 
	= Q_{k\ast}
	= \left[\sum\limits_v\frac{Q_{v\ast}}{Q_{k\ast}}\right]^{-1}
	= \left[\sum\limits_v\frac{Q_{kv}/Q_{k\ast}}{Q_{vk}/Q_{v\ast}}\right]^{-1}
	= \left[\sum\limits_v\frac{\pbtr{kv}}{\pbtr{vk}}\right]^{-1}.
	$$
	Indeed
	$$
	\sum\limits_k Q_{k\ast}\pbtr{kv} 
	= \sum\limits_k Q_{k\ast}\frac{Q_{kv}}{Q_{k\ast}} 
	= \sum\limits_k Q_{kv} = Q_{v\ast}.
	$$

	\item The expected degree of node $i$ in block $k$
	is
	$$
	\E[\deg{i}] 
	= \sum_{j} \theta_i \theta_j B_{Z_i, Z_j}
	= \theta_i \sum_\ell \sum_{j \in V_\ell} \theta_j B_{k\ell}
	= \theta_i B_{k\ast}.
	$$
	Hence the block-wise mean expected degree over block $k$
	is
	$$
	\mdegb{k} 
	= \frac{1}{N_k} \sum_{i \in V_k} \theta_i B_{k\ast}
	= \frac{B_{k\ast}}{N_k}.
	$$
	
	\item Combining the two observations above, we get
	$$
	\frac{\pstat{k}}{\mdegb{k}}
	= \frac{N_k}{\sum_k B_{k\ast}}.
	$$
	Because the denominator is constant, we have finally
	$$
	\alpha_k 
	= \frac{\pstat{k}/\mdegb{k}}{\sum_\ell \pstat{\ell}/\mdegb{\ell}}
	= \frac{N_k}{N}.
	$$
	
\end{enumerate}
Therefore, by~\eqref{eq:stratified-formulas}, 
the population mean $\mu_{\text{true}}$ can be re-written as 
$$
\mu_{\text{true}} 
= \sum_k \alpha_k\, \mu_k.
$$
From this it follows that, to estimate $\mu_{\text{true}}$, it suffices to estimate the block-wise mean $
\mu_k$, the block-wise expected mean degree $\mdegb{k}$, and the stationary distribution $\pstat{k}$ of $\pbtrmat^{B}$, for each block $k$.  We estimate them with  $\emfunc{k}$, $\edeg{k}$, and $\estat{k}$, respectively---leading to the PS estimator in~\eqref{eq: s-vh}.  In the proof of Theorem \ref{thm: consistency} below, we analyze the accuracy of these estimators (see Claims~\ref{claim:estat-conc}, \ref{claim:degrees-conc} and \ref{claim:emfunc-conc}). 



\section{Main theoretical result}\label{sec: theory}

In this section, we show that the PS estimator defined 
in \eqref{eq: s-vh} has error $O(\sqrt{\log n/n})$ with high probability
when the social network is distributed
under a dense DC-SBM. 

\begin{theorem}[Main result]   
\label{thm: consistency}
Suppose the social network $G = (V,E)$ of size $N$ is distributed
according to the DC-SBM with $K$ blocks of respective sizes 
$N_1,\ldots,N_k$ and parameters $B\in \mathbbm{R}_{+}^{K\times K}$ 
and $\theta \in \mathbbm{R}_+^N$. Suppose $\tree$ is a sampling tree of 
size $n \leq N$. Let $y \in  \mathbbm{R}_+^N$ be the variable of interest. 
Assume that there are universal constants $0 < c_- < c_+ < +\infty$ and 
$0 < c_y, c_d < +\infty$ independent of $N$ and $n$ such that
the following assumptions hold:
	\begin{enumerate}[label=(\alph*)]
	\item\label{ass: balanced} [Linear-sized blocks] $c_- N \leq N_k \leq c_+ N$ for all k;
	\item \label{ass: dense} [Dense graph] $c_- N^2 \leq B_{uv} \leq c_+ N^2$ for all blocks $u,v$;
	\item\label{ass: dense2} [Degree homogeneity] $c_- N^{-1} \leq \theta_i \leq c_+ N^{-1}$ for all nodes $i\in G$;
	\item\label{ass: bounded_y} [Bounded variables] $0 \leq y(i) \leq c_y$ for all nodes $i \in G$;
	\item\label{ass: degree} [Limited referrals] The maximum degree of $\tree$ is less than or equal to $c_d$.
\end{enumerate}
 Then, for any $\eps, \eps' >0$, there exists a constant $c>0$ (not depending on $n, N$) such that, with probability $1-\eps$ over the choice of $G$, the following holds.  
 For any $(\tree,P)$-walk on $G$ 
 the PS estimator defined in \eqref{eq: s-vh} 
 satisfies
$$|\ps-\mu_{\mathrm{true}}|\leq c\sqrt{\frac{\log n}{n}},$$
with probability $1-\eps'$.
\end{theorem}


A direct consequence of Theorem \ref{thm: consistency} is that the bias and standard deviation 
decay at rate $O(n^{-1/2})$ up to log factors.  This does not hold for the traditional VH estimator \cite{heckathorn1997respondent}, since its standard deviation decays at a rate slower than $O(n^{-1/2})$ \cite{rohe2015network}, which we also show by example in Appendix \ref{app:negative}.  For the recent GLS-based estimators proposed in \cite{roch2017generalized}, it is shown that their standard deviation decays at rate $O(n^{-1/2})$ as $n$ goes to infinity for a fixed network size, but no finite size
guarantees are provided. 

Assumptions \ref{ass: dense} and \ref{ass: dense2} require the graph to be dense.  In the following section, we show through simulations that the PS estimator also works well on sparse graphs.

\section{Simulations}\label{sec: simulation}
This section compares the PS estimator to the  VH and fGLS estimators on simulated networks (in Section \ref{sec: sim}) as well as social networks collected by the National Longitudinal Study of Adolescent Health (Add Health Networks) (in Section \ref{sec: addhealth}), both with simulated RDS samples.  In both cases, the PS estimator has smaller variation than the VH and fGLS estimators.

\subsection{Simulated Networks}\label{sec: sim}

We simulated 100 random social networks by DC-SBM with $10^5$ nodes, expected average degree $100$, and $K = 2$ blocks with the same sizes.  The stochastic matrix $B$ was chosen proportional to $$\left(\begin{matrix} 
0.95 &0.05\\
0.05 &0.95
\end{matrix}\right).$$  We simulated the binary outcomes to be perfectly aligned with one of the block labels.  

On each social network, we generated RDS samples by link tracing without replacement.  We randomly sampled the seed proportionally to the node degree.  Then, for each participant $\tau$ in the sample, we recruited $R_{\tau}\in\mathbbm{N}$ number of friends, where $R_{\tau} \stackrel{\text{iid}}{\sim} \text{Poi}(2)$.  The recruiting process stopped when there were 1000 participants in the RDS sample.  If it terminated before recruiting 1000 participants, then we re-started the recruiting process. We generated 200 different RDS samples on each network.  For each RDS sample, we computed the VH, fGLS, and PS estimators.  On each network, we computed the absolute bias, standard deviation, and RMSE of the 200 estimators of each type.   In the simulations, we computed the fGLS estimator as in \cite{roch2017generalized}, which re-weights the outcome $Y$ to adjust for the sampling bias.  

Figure \ref{fig: sim_main} shows that the PS estimator has smaller variation than the VH and fGLS estimators in terms of absolute bias, standard deviation, and RMSE.

\begin{figure}[H]\centering 
    \includegraphics[width=1\columnwidth]{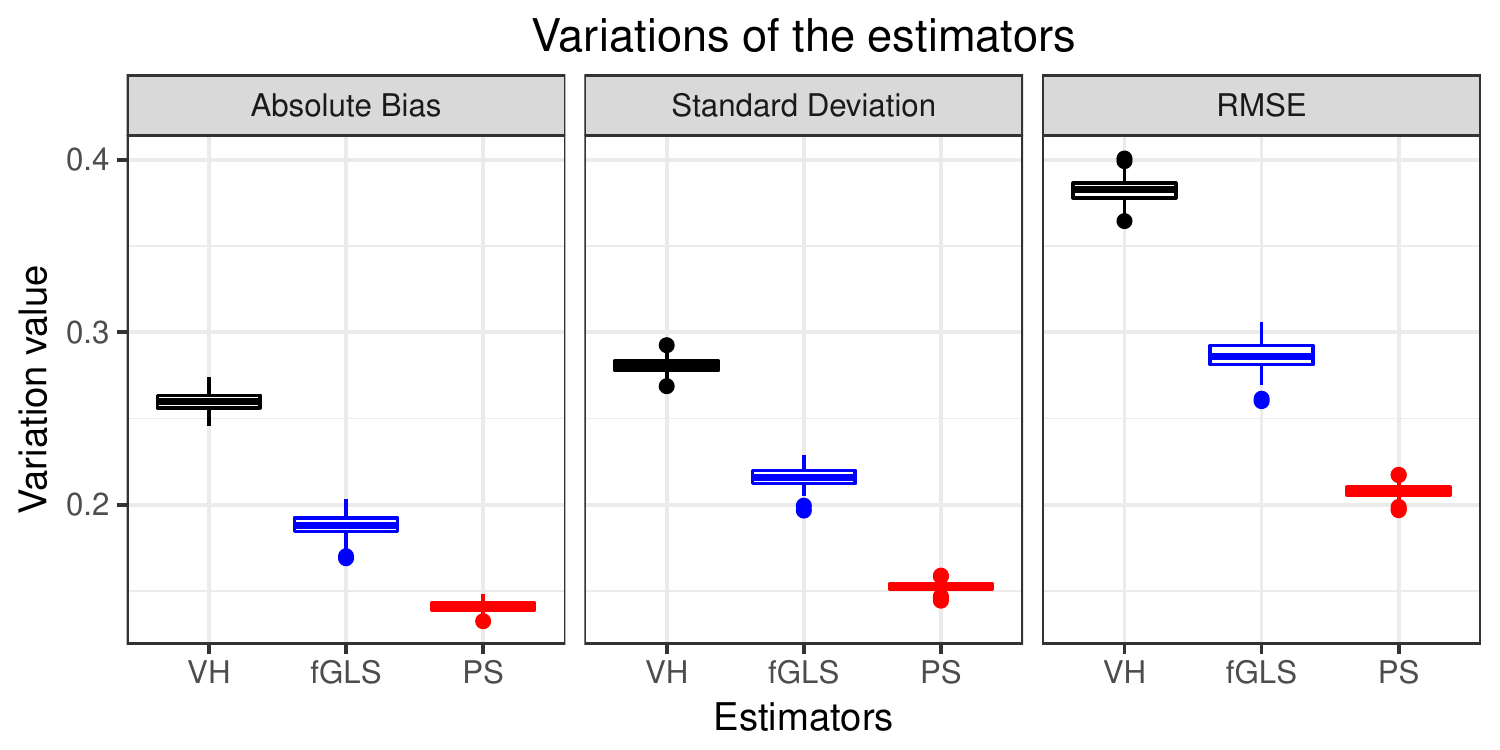}
	\caption{\textbf{Comparisons on the Simulated Networks.} The figures present the variations for the VH, fGLS and PS estimators on each of the 100 simulated networks. Each panel corresponds to a different variation, including absolute bias, standard deviation, and RMSE. Each data points represents for a variation value for a type of estimator on a network.  In each panel, for each type of estimator, we use a corresponding box plot to show the distribution of the variation values of the 100 simulated networks. 
	}
	\label{fig: sim_main}
\end{figure}

Note that there are some factors that may affect the performance of the estimators, such as (1) bottlenecks in the social network (2) the alignment of the block labels $z(i)$ with the variable of interest $y(i)$, and (3) the network density, etc.   We explored these factors and how they affected the performance of the estimators.  More explorations on other factors including network sizes and sample sizes are in Section \ref{sec: sim_factor_app} in the appendix.  The following simulations in Figure \ref{fig: sim_signal}, \ref{fig: sim_alignment} and \ref{fig: sim_density} have the same setting as in Figure \ref{fig: sim_main}, except that the values of the corresponding factor are made to vary.
 
\paragraph{Bottleneck}
Bottlenecks exist when there are much fewer connections across different blocks than within blocks.  Recall that, in the DC-SBM, the stochastic block matrix $B$ shows the average number of links between any two blocks.  We simulated the stochastic block matrix such that, 
$$B \propto \begin{pmatrix} 
 p & q \\
 q & p 
 \end{pmatrix},$$ 
 with $p + q = 1$ for identification.  We refer to the difference $p-q$ as the bottleneck strength. 
  With a larger bottleneck strength, there are more connections within blocks and fewer connections across blocks. 
  When there is no bottleneck (strength is zero), there is only one block in the network.  Figure \ref{fig: sim_signal} shows that the PS estimator has smaller variation than the fGLS and VH estimators, especially when there exists a bottleneck.  In particular, the PS estimator appears to reduce the seed bias and standard deviation 
  caused by bottlenecks much better than the fGLS and VH estimators.
 
  \begin{figure}[H]\centering 
 	\includegraphics[width=1\columnwidth]{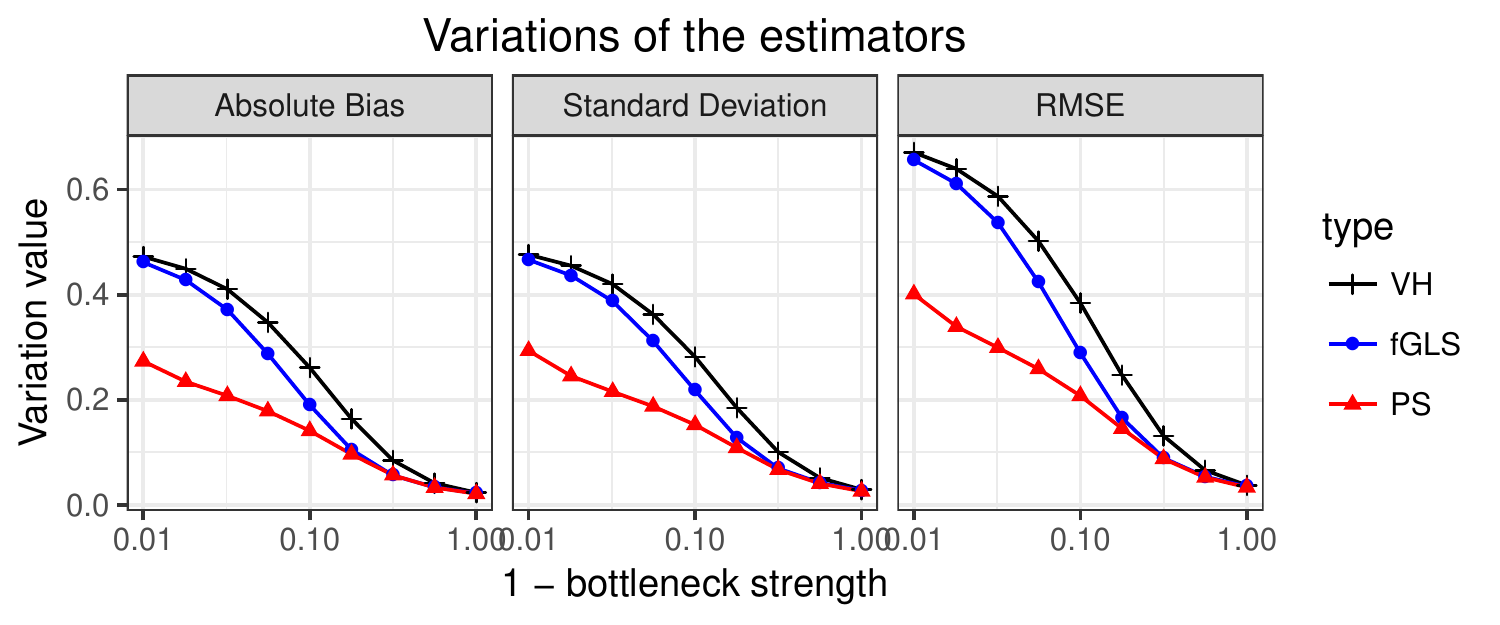}
 	\caption{\textbf{Comparisons on the simulated networks with different bottleneck strengths} }
 	\label{fig: sim_signal}
 \end{figure}
 
\paragraph{Alignment} We capture the alignment of the block labels and the variable of interest by the difference of the block-wise means of the variable of interest, i.e. $|\mu_1-\mu_2|$ with $K = 2$ blocks.  
Figure \ref{fig: sim_alignment} shows that the fGLS and PS estimators exhibit the largest improvement over the VH estimator when the block label perfectly aligns with the variable of interest (i.e., the alignment is $1$).  The three estimators perform equally well when the block-wise means of the variable of interest are equal (i.e., the alignment is $0$).  When the block label partially aligns with the variable of interest (i.e., the alignment is strictly between $0$ and $1$), the fGLS and VH estimators exhibit similar variation, but the PS estimator has smaller variation when the block-wise difference is over 0.4.
 
\begin{figure}[H]\centering 
	\includegraphics[width=1\columnwidth]{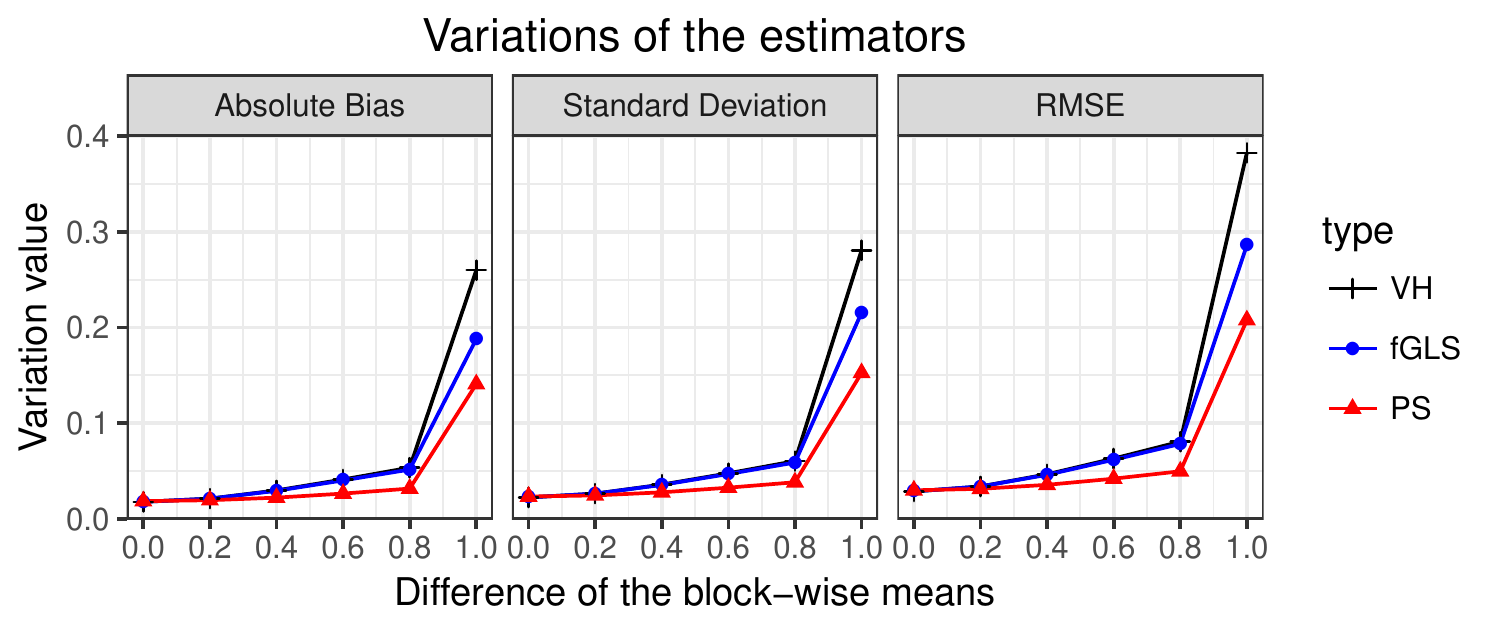}
	\caption{\textbf{Comparisons on the simulated networks with different alignments of the block labels} }
	\label{fig: sim_alignment}
\end{figure}

\paragraph{Network density}  We use the expected average degree of the network to quantify the network density.  Though Theorem \ref{thm: consistency} requires the networks to be dense enough, Figure \ref{fig: sim_density} shows that the estimators perform similarly on sparse networks.

\begin{figure}[H]\centering 
	\includegraphics[width=1\columnwidth]{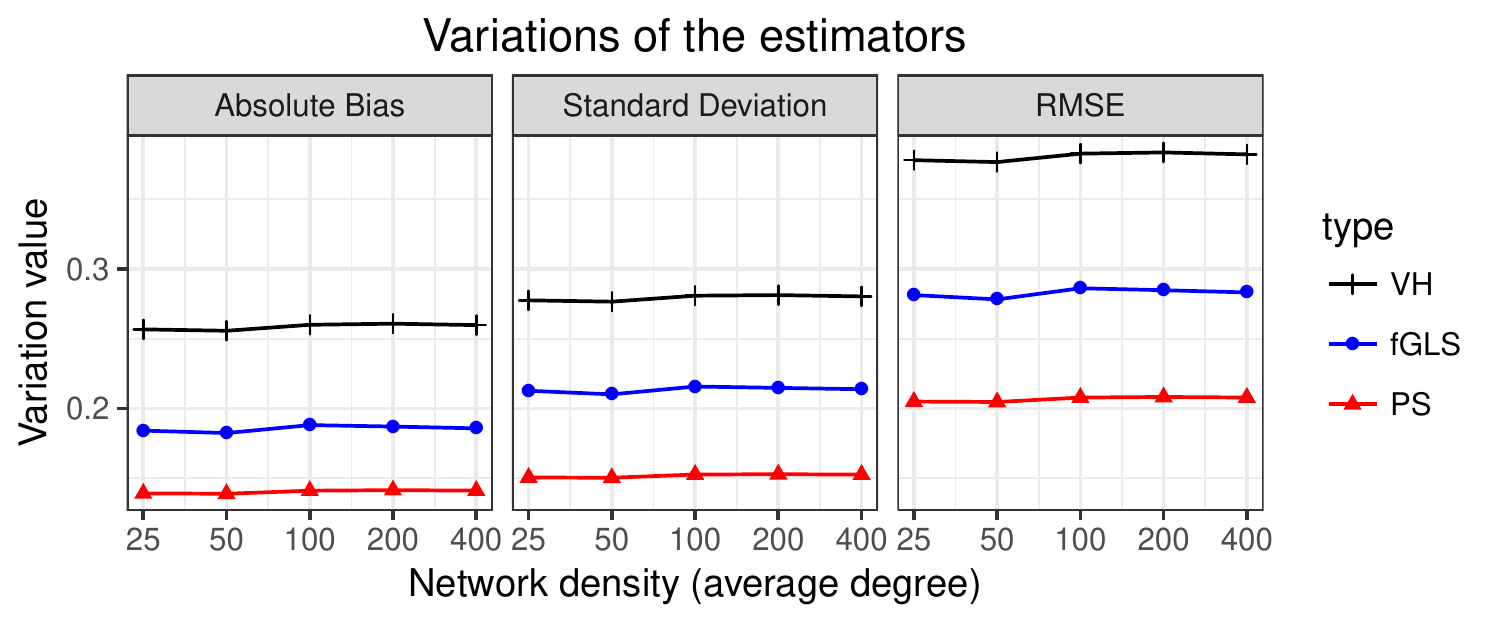}
	\caption{\textbf{Comparisons on the simulated networks with different density} }
	\label{fig: sim_density}
\end{figure}

\subsection{Add Heath Networks}\label{sec: addhealth}

In this section, we consider RDS simulations obtained by tracing contacts in social networks collected in the National Longitudinal Study of Adolescent Health (Add Health Networks).  This study collected a nationally represented sample of adolescents from grade 7 to 12 in the United States in the 1994-1995 school year. The sample covers 84 pairs of middle and high schools in which students nominated of up to five male and five female friends in their middle or high school network  (\cite{harris2011national}). In this analysis, we symmetrized all contacts to create a social network, and we restricted each network to its largest connected component.  These networks
were previously studied in \cite{goel2010assessing},  \cite{baraff2016estimating}, and \cite{roch2017generalized}.

We restricted our analysis to the 25 Add Heath networks with over 1000 nodes.  On each network, we simulated 200 different RDS samples, each with 500 participants.  On each RDS sample, we computed the VH, fGLS, and PS estimators.  In the simulation, we randomly sampled seed nodes proportional to node degrees.  We computed the absolute bias, RMSE, and standard deviation of the estimators on each network.  In the analysis we used the school label (middle school or high school) as the outcome and the grade label (7-12) as the block labels.  

The recruitment process was similar to that in Section \ref{sec: sim}, but without replacement.  In this case, each person could be recruited no more than once.  For each participant $\tau$, if they had fewer number of unrecruited friends than $R_{\tau}$, then we recruited all of their unrecruited friends.  

Figure \ref{fig: add_health} shows the variation of the estimators.  Overall, the PS estimator has substantially smaller variation than the fGLS and VH estimators.

\begin{figure}[H]\centering 
		\includegraphics[width=1\columnwidth]{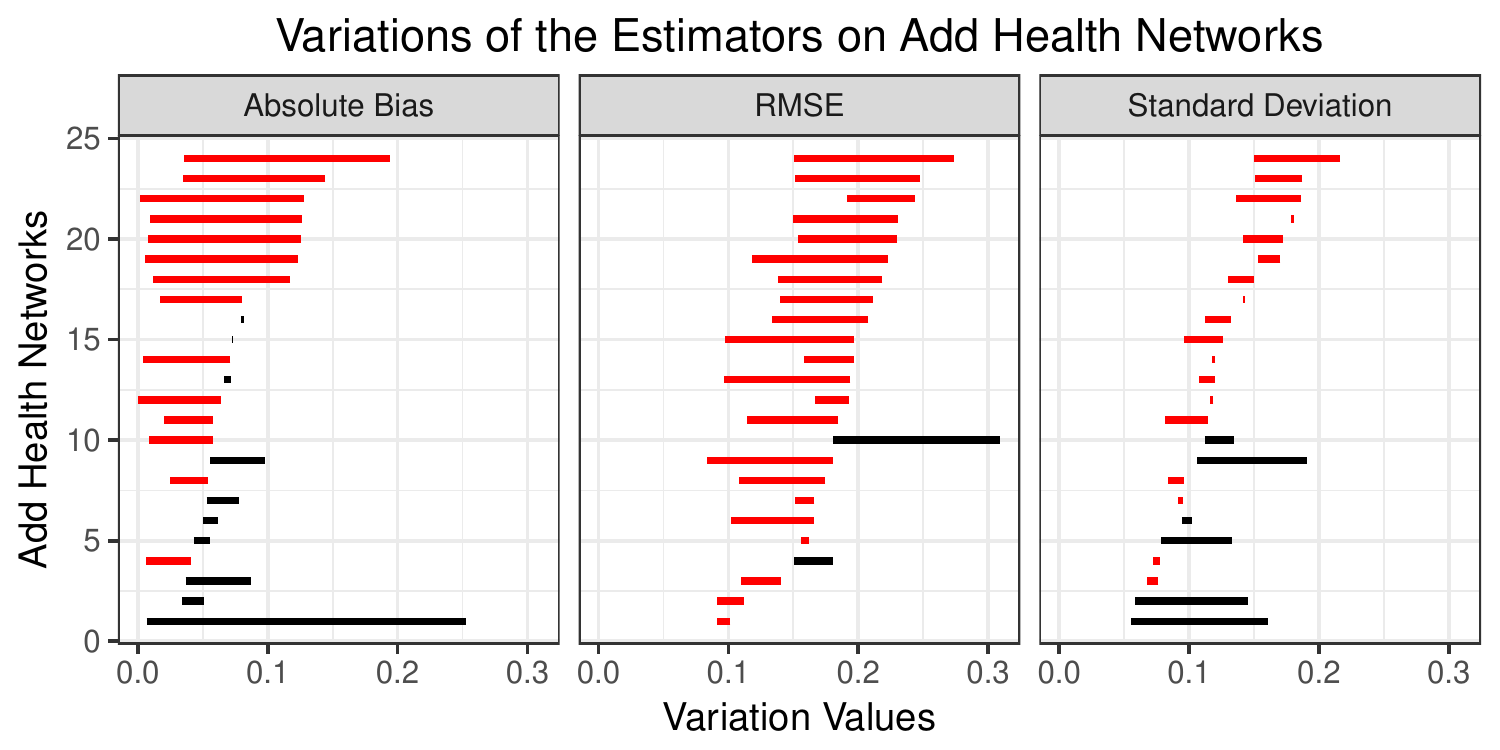}
	\caption{\textbf{Comparisons on the Add Health Networks.} The figures present the variations for the VH/fGLS and PS estimators on each of the 25 Add Health Networks. Each panel corresponds to a different variation, including absolute bias, RMSE, and standard deviation. In each panel, the horizontal axis corresponds to the variation value and the vertical axis corresponds to different networks, ordered by variation value of the baseline (VH/fGLS) estimator.  The baseline estimator is the one between VH and fGLS estimators with the smaller variation value, i.e. the better between VH and fGLS estimators.  Each line connects the variation value for the baseline estimator to the variation value of the PS estimator. If the line is red, then the PS estimator has a smaller variation.}
	\label{fig: add_health}
\end{figure}

\section{Discussion}\label{sec: discuss}
RDS has been widely used in studying marginalized populations.  But the estimators derived from RDS samples have suffered from high variance.  This is due to two related issues (1) the complicated network dependence of the RDS samples, and (2) seed bias caused by bottlenecks.  In this paper, we introduced post-stratification to RDS and provided a novel estimator.  Our easy-to-compute PS estimator reduces seed bias.  We derived some theoretical results for the PS estimator, showing its bias and standard deviation decay at $O(n^{-1/2})$ (up to log factors) under the degree-corrected stochastic block model.  This is the first estimator with such guarantees.  Though we require the networks to be dense in theory, we showed through simulations that the estimator performs similarly on sparse networks.

One future direction is how to select the block labels in practice. In \cite{roch2017generalized}, an approach for selecting block labels using eigenvalues of the block-wise transition matrix $\hat{Q}$ is proposed.  Further discussions on this issue would be helpful to apply the PS (and fGLS) estimators.

\bibliography{foo}

\begin{thebibliography}{17}
\providecommand{\natexlab}[1]{#1}
\providecommand{\url}[1]{\texttt{#1}}
\expandafter\ifx\csname urlstyle\endcsname\relax
  \providecommand{\doi}[1]{doi: #1}\else
  \providecommand{\doi}{doi: \begingroup \urlstyle{rm}\Url}\fi

\bibitem[Heckathorn(1997)]{heckathorn1997respondent}
Douglas~D Heckathorn.
\newblock Respondent-driven sampling: a new approach to the study of hidden
  populations.
\newblock \emph{Social problems}, 44\penalty0 (2):\penalty0 174--199, 1997.

\bibitem[Malekinejad et~al.(2008)Malekinejad, Johnston, Kendall, Kerr, Rifkin,
  and Rutherford]{malekinejad2008using}
Mohsen Malekinejad, Lisa~Grazina Johnston, Carl Kendall, Ligia Regina
  Franco~Sansigolo Kerr, Marina~Raven Rifkin, and George~W Rutherford.
\newblock Using respondent-driven sampling methodology for hiv biological and
  behavioral surveillance in international settings: a systematic review.
\newblock \emph{AIDS and Behavior}, 12\penalty0 (1):\penalty0 105--130, 2008.

\bibitem[Johnston(2013)]{johnston2013introduction}
LG~Johnston.
\newblock Introduction to hiv/aids and sexually transmitted infection
  surveillance: Module 4: Introduction to respondent driven sampling.
\newblock \emph{World Health Organization}, 2013.

\bibitem[White et~al.(2015)White, Hakim, Salganik, Spiller, Johnston, Kerr,
  Kendall, Drake, Wilson, Orroth, et~al.]{white2015strengthening}
Richard~G White, Avi~J Hakim, Matthew~J Salganik, Michael~W Spiller, Lisa~G
  Johnston, Ligia Kerr, Carl Kendall, Amy Drake, David Wilson, Kate Orroth,
  et~al.
\newblock Strengthening the reporting of observational studies in epidemiology
  for respondent-driven sampling studies:“strobe-rds” statement.
\newblock \emph{Journal of clinical epidemiology}, 68\penalty0 (12):\penalty0
  1463--1471, 2015.

\bibitem[Rohe(2015)]{rohe2015network}
Karl Rohe.
\newblock Network driven sampling; a critical threshold for design effects.
\newblock \emph{arXiv preprint arXiv:1505.05461}, 2015.

\bibitem[Roch and Rohe(2017)]{roch2017generalized}
Sebastien Roch and Karl Rohe.
\newblock Generalized least squares can overcome the critical threshold in
  respondent-driven sampling.
\newblock \emph{arXiv preprint arXiv:1708.04999}, 2017.

\bibitem[Goel and Salganik(2009)]{goel2009respondent}
Sharad Goel and Matthew~J Salganik.
\newblock Respondent-driven sampling as markov chain monte carlo.
\newblock \emph{Statistics in medicine}, 28\penalty0 (17):\penalty0 2202--2229,
  2009.

\bibitem[Salganik and Heckathorn(2004)]{salganik2004sampling}
Matthew~J Salganik and Douglas~D Heckathorn.
\newblock Sampling and estimation in hidden populations using respondent-driven
  sampling.
\newblock \emph{Sociological methodology}, 34\penalty0 (1):\penalty0 193--240,
  2004.

\bibitem[Volz and Heckathorn(2008)]{volz2008probability}
Erik Volz and Douglas~D Heckathorn.
\newblock Probability based estimation theory for respondent driven sampling.
\newblock \emph{Journal of official statistics}, 24\penalty0 (1):\penalty0 79,
  2008.

\bibitem[Benjamini and Peres(1994)]{benjamini1994markov}
Itai Benjamini and Yuval Peres.
\newblock Markov chains indexed by trees.
\newblock \emph{The annals of probability}, pages 219--243, 1994.

\bibitem[Horvitz and Thompson(1952)]{horvitz1952generalization}
Daniel~G Horvitz and Donovan~J Thompson.
\newblock A generalization of sampling without replacement from a finite
  universe.
\newblock \emph{Journal of the American statistical Association}, 47\penalty0
  (260):\penalty0 663--685, 1952.

\bibitem[Karrer and Newman(2011)]{karrer2011stochastic}
Brian Karrer and Mark~EJ Newman.
\newblock Stochastic blockmodels and community structure in networks.
\newblock \emph{Physical review E}, 83\penalty0 (1):\penalty0 016107, 2011.

\bibitem[Harris(2011)]{harris2011national}
Kathleen~Mullan Harris.
\newblock The national longitudinal study of adolescent health: Research
  design.
\newblock \emph{http://www. cpc. unc. edu/projects/addhealth/design}, 2011.

\bibitem[Goel and Salganik(2010)]{goel2010assessing}
Sharad Goel and Matthew~J Salganik.
\newblock Assessing respondent-driven sampling.
\newblock \emph{Proceedings of the National Academy of Sciences}, 107\penalty0
  (15):\penalty0 6743--6747, 2010.

\bibitem[Baraff et~al.(2016)Baraff, McCormick, and
  Raftery]{baraff2016estimating}
Aaron~J Baraff, Tyler~H McCormick, and Adrian~E Raftery.
\newblock Estimating uncertainty in respondent-driven sampling using a tree
  bootstrap method.
\newblock \emph{Proceedings of the National Academy of Sciences}, page
  201617258, 2016.

\bibitem[Hoeffding(1963)]{hoeffding1963probability}
Wassily Hoeffding.
\newblock Probability inequalities for sums of bounded random variables.
\newblock \emph{Journal of the American statistical association}, 58\penalty0
  (301):\penalty0 13--30, 1963.

\bibitem[Motwani and Raghavan(1995)]{MotwaniRaghavan:95}
Rajeev Motwani and Prabhakar Raghavan.
\newblock \emph{Randomized algorithms}.
\newblock Cambridge University Press, Cambridge, 1995.
\newblock ISBN 0-521-47465-5.

\end{thebibliography}
\bibliographystyle{unsrtnat}

\appendix

\section{More Simulations}\label{sec: sim_factor_app}

In this section, we explore how network sizes and sample sizes affect the performances of RDS estimators.  The simulation settings are the same as in Section \ref{sec: sim}.  Figure \ref{fig: sim_networksize} shows the estimators perform similarly with different sizes of networks.  Figure \ref{fig: sim_samplesize} shows the RDS estimators have smaller variation with larger sample sizes.  

\begin{figure}[H]\centering 
	\includegraphics[width=1\columnwidth]{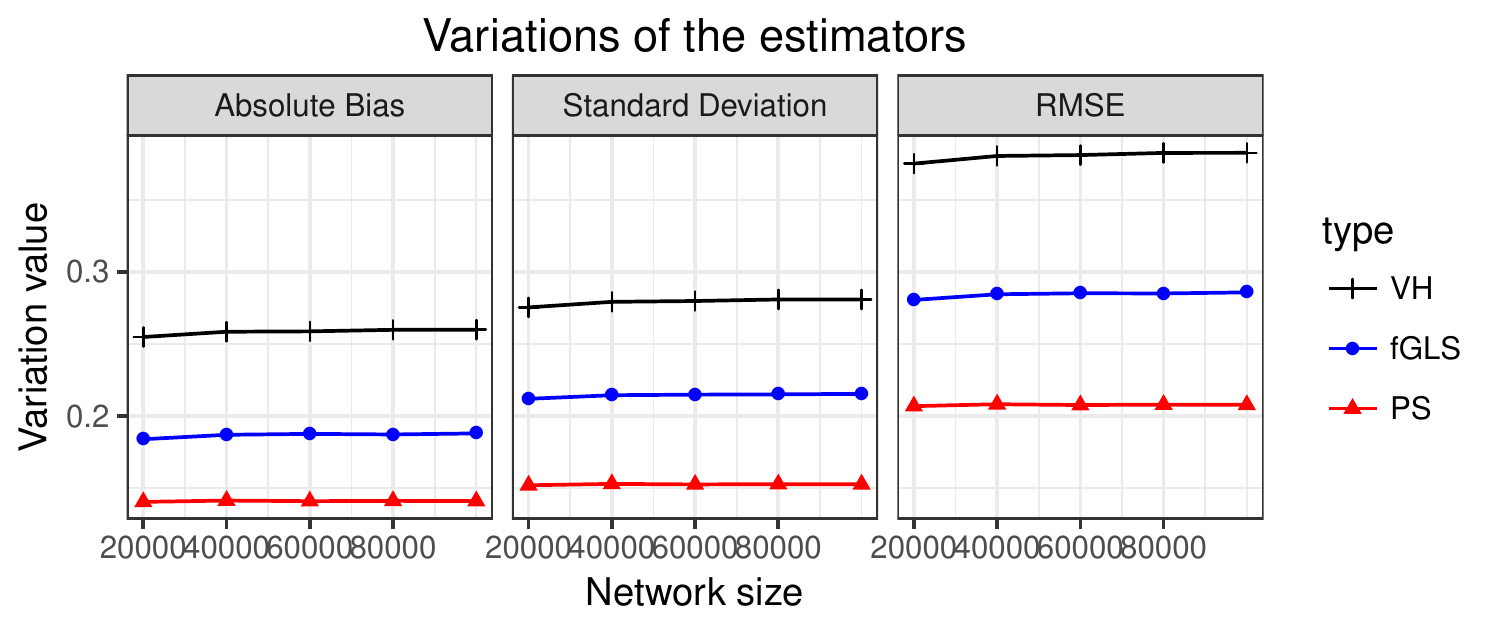}
	\caption{\textbf{Comparisons on the simulated networks with different sample sizes.} In the simulations, we control the densities for the networks to be similar.  For each network with size $N$, we set the expected average degree as $\lfloor\sqrt{N}/3\rfloor$, where $\lfloor c \rfloor$ denotes the integer part of $c$ for any constant $c\in \mathbbm{R}$. }
	\label{fig: sim_networksize}
\end{figure}

\begin{figure}[H]\centering 
	\includegraphics[width=1\columnwidth]{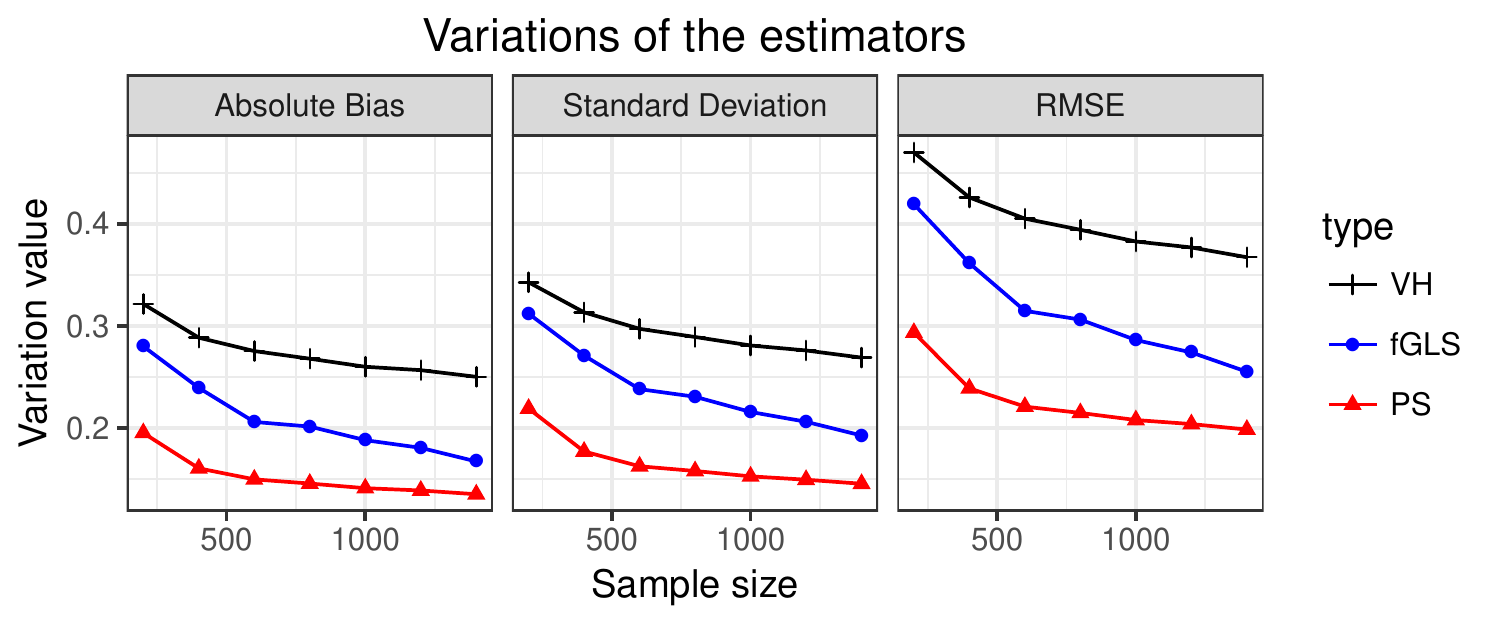}
	\caption{\textbf{Comparisons on the simulated networks with different sample sizes.} }
	\label{fig: sim_samplesize}
\end{figure}



\section{Proof of the main theorem}\label{sec: proof}

\subsection{Notation}

	For each node $i \in V$, we denote its neighborhood in the social network $G$ as
	$
	\neighb(i)
	=
	\left\{
	j \in V\,:\,\{i,j\} \in E
	\right\}
	$
	and its neighborhood within block
	$k$ as
	$
	\neighb(i;k)
	=
	\left\{	
	j \in \ve_k\,:\,\{i,j\} \in \ed
	\right\}.
	$
	We denote by $d(i;k) = |\neighb(i;k)|$ the size of the latter.  
	The degree of node $i$ is denoted $d_i = |\neighb(i)|$ and we have $d_i  = \sum\limits_{k}d(i;k)$. 

    While the RDS sampling procedure is a random walk on the social network $G$, under a dense DC-SBM our analysis relies on establishing an approximation of the process by a ``population-level'' random walk on blocks.  We define the block transition probability at node $i\in V$ by
    \begin{equation}\label{def: transition_block}
    p_{uv}(i) = \P[Z_{\tau} = v\,|\, X_{\tau'} = i, z(i) = u] =
	\frac{d(i;v)}{\sum_{k} d(i;k)}
	\ind{\{
	z(i) = u
	\}},
    \end{equation}
    for any blocks $u,v$ and any sample $\tau \in \mathbb{T}$. Recall that, for any sample $\tau\in\mathbbm{T}$, we denote its parent as $\tau'$. 
	
	Under our assumptions, 
	$B_{uv}$ is the expected number of edges between blocks $u\neq v$; indeed 
	$$
	\sum_{i\in V_u, j \in V_v} 	\P[\{i,j\} \in E]
	= \sum_{i\in V_u, j \in V_v} \theta_i \theta_j B_{uv}
	= B_{uv} \sum_{i\in V_u} \theta_i \sum_{j \in V_v} \theta_j
	= B_{uv}.
	$$
	Recalling the matrix $Q = B/m$, where $B$ is the matrix in the definition of the DC-SBM model \eqref{def: dc-sbm} and $m = \ind^T B \ind$, the population block transition probability is given by
	\begin{equation}
	\label{eq:pbtr-def}
	\pbtr{uv} 
    = \frac{\aff{uv}}{\aff{u\ast}}
	= 
	\frac{Q_{uv}}{Q_{u\ast}},
	\end{equation}
    for any two blocks $u,v$.
    We refer to 
	\begin{equation}
	\label{eq:pstatvec}
	\pbtrmat^{B}
	=
	\left(
	\pbtr{uv}
	\right)_{uv}.
	\end{equation}
	as the population transition matrix on blocks.
	Recall that its unique 
	stationary distribution is $\pstatvec = (\pstat{k})_k$.

For each block $k\in\{1,\dots, K\}$, 
$n_k =  |\mathbbm{T}_k|$. 
We also define the number of referrals from block $k$ to be
	\begin{equation}\label{def: n_k'}
	n_{k'} = \sum\limits_{\tau\in\mathbbm{T}}\ind\{Z_{\tau'} = k\}.
	\end{equation} 
	For any two blocks $u,v\in\{1,\dots, K\}$, we define the number of referrals between block $u$ and block $v$ as 
	\begin{equation}\label{def: n_u'v}
	n_{u'v} = \sum\limits_{\tau\in\mathbbm{T}}\ind\{Z_{\tau'} = u, Z_{\tau} = v\}.  
	\end{equation}
	Note that $\hat{Q}_{uv} = n_{u'v}/n$ and $\hat{P}^B_{uv} = n_{u'v}/n_{u'}$ 
The elements of the estimated block transition matrix $\hat{P}^B$ in \eqref{def: hat_QR} can be rewritten as $\hat{P}^B_{uv} = \etrans_{u'v}/\etrans_{u'}$.  We use $\epbtr{uv}$ to denote these quantities, i.e.,
$$
	\epbtr{uv}
	:=
	\frac{ \etrans_{u'v} }
	{ \etrans_{u'} }.
$$

To summarize, for any blocks $u,v$, the quantities $p_{uv}(i)$, $\tilde{p}_{uv}$ and $\hat{p}_{uv}$ represent respectively the block transition probability at node $i\in G$, the population block transition probability, and the estimated block transition probability.

\subsection{Proof}
The proof of Theorem~\ref{thm: consistency} follows from a series
of claims. We begin with a sketch of the proof in this section.  
\begin{enumerate}
\item Under the dense DC-SBM, random walk is mixing fast within each block (Claims \ref{claim:degreetwos} and \ref{claim:two-step-mixing}). 
This plays a key role in estimating block-wise means, for which we use the VH estimator (Claims \ref{claim:emfunc-conc}, \ref{claim: help_degree}, and \ref{claim:close}).

\item To estimate block proportions, we use the stationary distribution of block-wise transition matrix, which is the main, non-trivial contribution of this work. Indeed, the standard empirical frequency gives an estimate with much larger variance (see Section \ref{app:negative}). Instead we estimate the transition matrix between blocks, which is a ``more local'' quantity in the sense that it is not affected strongly by the seed, and compute its stationary distribution.  As a result, block-wise transition probabilities are highly concentrated around their true value under the Markov chain on the blocks; their stationary distributions are also close to each other (Claim \ref{claim:block-transition}, Claim \ref{claim:estat-conc}).  
\end{enumerate}
Note that there are two sources of randomness, the social network $G$ and the $\mathbbm{T}$-indexed random walk.  Claims \ref{claim:degrees}-\ref{claim:two-step-mixing} are concerned with the randomness of $G$, while Claims \ref{claim:estat-conc}-\ref{claim:close} deal with the random walk.
\begin{center}
\begin{tikzpicture}[
 every matrix/.style={ampersand replacement=\&,column sep=3cm,row sep=0.2cm},
c/.style={draw,thick,rounded corners,fill=yellow!20,inner sep=.2cm},
to/.style={->,>=stealth',shorten >=1pt,thick},
every node/.style={align=center}]

  \matrix{
       \ \node[c] (mu1) { $\emfunc{k}$
       };
      \& \node[c] (mu2) { $\mfunc_k$
   }; \\
       \ \node[c] (pi1) { $\estat{k}$ 
       };  
       \&  \node[c] (pi2) { $\pstat{k}$ 
       };  \\
        \node[c] (d1) { $H^{(\delta_k)}$ 
       };  
       \&  \node[c] (d3) { $\delta_k^{(B)} = N_k^{-1}B_{k\ast}$ 
       }; \\ 
       \ \node[c] (u1) { $\mvh$ 
       };  
       \&  \node[c] (u2) { $\mu_{\text{true}}$ 
       };  \\
  };

   \draw[to] (mu1) -- 
  node[midway,above] {Claim \ref{claim:close}} (mu2);
   \draw[to] (pi1) -- node[midway,above] {Claim \ref{claim:estat-conc} 
} (pi2);
 \draw[to] (d1) -- node[midway,above] {Claim \ref{claim:degrees-conc} 
} (d3);
\draw[to] (u1) -- node[midway,above] {Combine above
} (u2);



 \end{tikzpicture}
\end{center}

Throughout, $\eps > 0$ is as in the statement of the theorem.

\paragraph{High-probability properties of the social network}
We first use standard concentration inequalities to control the degrees
of $G$. 
Recall that under the DC-SBM the expectation of $\deg{}(i;v)$ is $\theta_i B_{uv}$.

\begin{claim}[Degrees are concentrated]
	\label{claim:degrees}
Under the DC-SBM, there exists $c_1 > 0$ (depending on $\eps$ but not on $N$) such 
	that,
with probability $1 - \eps/2$ over the choice of $\gr$, the following event holds:
simultaneously for 
all pairs of blocks $u,v$ and all nodes $i \in \ve_u$,
$$
\left|
\frac{\deg{}(i;v)}{\theta_i\aff{uv}}
- 1 
\right|
\leq
c_1 \sqrt{\frac{\log \nve}{\nve}}.
$$	
\end{claim}
\noindent We let $\eventdeg$ be the event in the claim.
\begin{proof}[Proof of Claim~\ref{claim:degrees}]
Fix blocks $u,v$ and $i \in \ve_u$. Under the DC-SBM, each node $j$ in block $v$ connects with node $i$ independently with probability $\theta_i \theta_j \aff{uv}$. Hence we can write
$\deg{}(i;v)$ as a sum of $\bsize{v}$ independent indicators, whose overall expectation is $\theta_i \aff{uv}$,
where we used that $\sum\limits_{j\in V_v}\theta_j = 1$.
By Hoeffding's inequality \cite{hoeffding1963probability}, for any constant $c_1'>1$, by choosing $c_1 > 0$ large enough
\begin{eqnarray}
&&\P\left[
\left|
\deg{}(i;v)
-  \theta_i\aff{uv}
\right|
>
c_1 \sqrt{\nve \log \nve}
\right]\nonumber\\
&&\leq
2
\exp\left(
- \frac{2  \left[
	c_1 \sqrt{\nve \log \nve}
	\right]^2}{\bsize{v}}
\right)\nonumber\\
&&\leq \nve^{-c_1'},\label{eq:deg-conc}
\end{eqnarray}
where we used that $\bsize{v} = \Theta(\nve)$
in the second inequality.
Taking a union bound over $u$, $v$ and $i$ gives
$$
\left|
\deg{}(i;v)
- \theta_i\aff{uv}
\right|
\leq
c_1 \sqrt{\nve \log \nve},
$$	
simultaneously for 
all $u,v$ and all $i \in \ve_v$ with probability
at least $1-K^2 \cdot \nve \cdot \nve^{-c_1'}$.  
Dividing by $\theta_i\aff{uv}$ and using $\theta_i = \Theta(\nve^{-1})$ for any node $i$ and $\aff{uw} = \Theta(\nve^2)$ for any blocks $u,w$, gives the result for appropriately
chosen $c_1, c_1' > 0$.
\end{proof}

The following claim will be useful to control the mixing rate within a block.
For any blocks $u, w, v$ and two distinct nodes $i\in \ve_u$, $j\in \ve_v$, we consider the number of two-edge paths from $i$ to $j$ in $\gr$ whose middle vertex is in block $w$, weighted by a quantity related to the
expected degree of the middle vertex under the DC-SBM: 
$$
\degtwow(i,j;w) = \sum\limits_{k\in V_w} \frac{1}{N\theta_k}
\ind\{k\in\neighb(i)\cap\neighb(j)\}.
$$
\begin{claim}[Two-edge paths]
	\label{claim:degreetwos}
	There exists $c_2 > 0$ such 
	that,
	with probability $1 - \eps/2$ over the choice of $\gr$, the following holds:
    simultaneously 
    for 
	all blocks $u, w, v$,
	and all $i\in \ve_u$, $j\in \ve_v$
    with $i \neq j$,
	$$
	\left|
	\frac{
		\degtwow(i,j;w)
	}
	{
		N^{-1}\theta_i\theta_j\aff{uw} \aff{wv}
	}
	-
	1 
	\right|
	\leq
	c_2 \sqrt{\frac{\log \nve}{\nve}}.
	$$
\end{claim}
\noindent We let $\eventdegtwo$ be the event in the claim.
\begin{proof}[Proof of Claim~\ref{claim:degreetwos}]
Fix blocks $u,w,v$, and nodes $i\not=j$.  By Claim \ref{claim:degrees}, we can choose $c_1''$ large enough such that 
\begin{equation}
\label{eq:degtwo-first-step}
\P\left[
\left|
\deg{}(i;w)
-  \theta_i\aff{uw}
\right|
>
c_1'' \sqrt{\nve \log \nve}
\right] \leq \nve^{-c_1'''},
\end{equation}
for some $c_1'''>1$.  

We treat the case where all blocks are distinct. The other cases are similar. Let $\mathcal{E}_{i,w}$ be the event that $\left|\deg{}(i;w) - \theta_i\aff{uw}\right|\leq c_1''\sqrt{\nve\log\nve}$ and note that $j\not\in\neighb(i;w)$.  
	Conditioned on $\mathcal{E}_{i,w}$, each of the $\deg{}(i;w)$ edges incident to $i$ and block $w$ has a corresponding endpoint $k\in V_w$ which itself connects to $j$---independently of all other such endpoints---with probability $\theta_k\theta_j\aff{wv}$.  
	Since $\degtwow(i,j;w)$ weighs this last edge by $(N\theta_k)^{-1}$, its expected contribution is $N^{-1}\theta_j\aff{wv}$. Moreover, the $\deg{}(i;w)$ possibly non-zero terms in the sum defining $\degtwow(i,j;w)$ are uniformly bounded by a constant by the assumption that $\theta_i = \Theta(N^{-1})$. Hence, we can apply 
	Hoeffding's inequality again, and by choosing $c_2'>1$ large enough we have 
	\begin{eqnarray}
	&& \P\left[
	\left|
	\degtwow(i,j;w) - N^{-1}\theta_j \aff{wv}\deg{}(i;w)
	\right|>c_2'\sqrt{\nve\log\nve}
	\,\middle|\, \mathcal{E}_{i,w}\right]\nonumber\\
	&& 
	\leq 
	2
	\exp\left(
	- \frac{2  \left[
		c_2' \sqrt{\nve \log \nve}
		\right]^2}{\deg{}(i;w)}
	\right)\nonumber\\
	&& 
	\leq 
	2
	\exp\left(
	- \frac{2  \left[
		c_2' \sqrt{\nve \log \nve}
		\right]^2}{\theta_i\aff{uw} + c_1'' \sqrt{\nve \log \nve}}
	\right)\nonumber\\
	&& 
	\leq \nve^{-c_2''},\label{eq:degtwo-second-step}
	\end{eqnarray}
	for some $c_2''>2$, where we used \eqref{eq:degtwo-first-step} in the second inequality
	and we used that $\theta_i = \Theta(N^{-1})$ and $\aff{uw} = \Theta(N^2)$ in the last inequality.

Combining~\eqref{eq:degtwo-first-step} and~\eqref{eq:degtwo-second-step},
and taking a union bound over $u$, $w$, $v$, $i$ and $j$ gives
$$
\left|
\degtwo(i,j;w)
-
N^{-1}\theta_i\theta_j\aff{uw}\aff{wv}
\right|
\leq
c_2''' \sqrt{\nve \log \nve},
$$
for a constant $c_2''' > 0$ chosen large enough.
Dividing by $N^{-1}\theta_i\theta_j\aff{uw}\aff{wv}$ and using again that
$\theta_i = \Theta(N^{-1})$ and $\aff{uv} = \Theta(N^2)$
gives the result,
for an appropriately chosen constant $c_2 > 0$.
\end{proof}

\paragraph{Properties of the walk} Before proving our main theorem,
we will also need some results about the behavior of simple random walk
on the network. We first show that, from any $i \in \ve_u$, the
probability of jumping to a vertex in block $v$ is close to the
population-level probability $\pbtr{uv}$.
\begin{claim}[Transitions between blocks]
	\label{claim:block-transition}
	There exists $c_3 > 0$ such that, conditioned on $\eventdeg$,
for any blocks $u,v$ and any $i \in \ve_{u}$	
	$$
\left|
\frac{\btr{i}{uv}}
{\pbtr{uv}}
-
1
\right|
\leq
c_3
\sqrt{
	\frac{\log \nve}{\nve}
}.
$$ 
\end{claim}
\begin{proof}
	Fix $u,v$ and $i \in \ve_{u}$.
	Recall 
	$$
	\btr{i}{uv}
	=
	\frac{\deg{}(i;v)}{\sum_{k} \deg{}(i;k)}
	\quad \text{ and } \quad
	\pbtr{uv} 
	= 
	\frac{\aff{uv}}{\aff{u\ast}}.
	$$
	Under $\eventdeg$,
\begin{eqnarray*}
\btr{i}{u,v}
\leq
\frac{\theta_i\aff{uv} \left(1 + c_1  \sqrt{\frac{\log \nve}{\nve}} \right)}{\sum_{k} \theta_i\aff{uk} \left(1 - c_1  \sqrt{\frac{\log \nve}{\nve}} \right)}
\leq
\pbtr{uv} 
\left(1 + c_3  \sqrt{\frac{\log \nve}{\nve}} \right),
\end{eqnarray*}
for a constant $c_3 > 0$ large enough.
A similar inequality holds in the opposite direction.
\end{proof}

The previous claim also implies that
any step has a probability bounded away from $0$
of landing in any block.
\begin{claim}[Landing in a block]
	\label{claim:landing}
There is $p_* \in (0,1)$ such that,
conditioned on $\eventdeg$,
for any blocks $u,v$ and any $i \in \ve_u$, we have
$$
\btr{i}{uv} 
\geq p_*,
$$
provided $\nve$ is larger than a sufficiently large constant.
\end{claim}
\begin{proof}
Let 
$$
0 <  p_* < \min_{u,v} \pbtr{uv}.
$$
The result then follows from Claim~\ref{claim:block-transition}.
\end{proof}

We next show that two steps of the walk are enough to mix within a block.  
\begin{claim}[Two steps suffice for within-block mixing]
	\label{claim:two-step-mixing}
For each sample $\tau\in\mathbbm{T}$, we denote its grandchildren as $\childtwo(\tau)$.  For a $(\mathbbm{T}, P)$-walk on $\gr$, 
there exists $c_4 > 0$ such that, on $\eventdeg$ and $\eventdegtwo$,
for all $\tau$ and $\tau_{**} \in \childtwo(\tau)$
$$
\left|
\frac{\P[X_{\tau_{**}} = j\,|\, X_{\tau} = i,\gr]}{ \theta_j \sum\limits_{k} \pbtr{uk} \pbtr{kv}}
-
1
\right|
\leq
c_4 
\sqrt{\frac{\log \nve}{\nve}},
$$
for all blocks $u,v$, and nodes $i \in \ve_u$,  $j \in \ve_v$
with $i \neq j$.
\end{claim}
\begin{proof}
To simplify
notation, the conditioning on $\gr$ is implicit throughout the proof.
Assume $\eventdeg$ and $\eventdegtwo$ hold. 
Fix blocks $u,v$ as well as nodes $i \in \ve_u$ and $j \in \ve_v$ with $i\neq j$.
Let $\tau_{**} \in \childtwo(\tau)$ and let $\tau_{*}$ be the ancestor of $\tau_{**}$ on $\tree$, which is necessarily a child of $\tau$.
Then, for some constants $c_4', c_4'' > 0$, using
$\eventdeg$ and $\eventdegtwo$
\begin{eqnarray*}
&&\P[X_{\tau'_{**}} = j\,|\, X_{\tau} = i]\\
&&= 
\sum_{t \in \ve}
\P[X_{\tau_{**}} = j\,|\, X_{\tau_{*}} = t ]\,
\P[ X_{\tau_{*}} = t \,|\, X_{\tau} = i]\\
&&= 
\sum_{t \in \neighb(i) \cap \neighb(j)}
\frac{1}{\deg{i}} \times
\frac{1}{\deg{t}}\\
&&\leq 
\sum_{t \in \neighb(i) \cap \neighb(j)}
\frac{1}
{\theta_t\aff{Z_t,\ast} \left(1 - c_1  \sqrt{\frac{\log \nve}{\nve}} \right)}
\times
\frac{1}
{\theta_i\aff{u\ast} \left(1 - c_1  \sqrt{\frac{\log \nve}{\nve}} \right)}\\
&&\leq 
 \left(1 + c_4'  \sqrt{\frac{\log \nve}{\nve}} \right)
\sum_{k} 
\frac{1}
{ \theta_i\aff{k\ast} \aff{u\ast} }
\sum_{t \in \neighb(i;k) \cap \neighb(j;k)}
\frac{1}
{ \theta_t}\\
&&= 
 \left(1 + c_4'  \sqrt{\frac{\log \nve}{\nve}} \right)
\sum_{k}\frac{ \nve \degtwow(i,j;k)}
{\theta_i\aff{k\ast} \aff{u\ast}}\\
&&\leq 
 \left(1 + c_4'  \sqrt{\frac{\log \nve}{\nve}} \right)
\sum_{k} \frac{\theta_i\theta_j\aff{uk}\aff{kv} \left(1 + c_2  \sqrt{\frac{\log \nve}{\nve}} \right)}
{\theta_i\aff{k\ast} \aff{u\ast}}\\
&&\leq 
\left(1 + c_4''  \sqrt{\frac{\log \nve}{\nve}} \right)
\theta_j  \sum_{k}  \pbtr{uk} \pbtr{kv},
\end{eqnarray*}
where recall that $\pbtr{uv}  = \aff{uv}/\aff{u\ast}$.
A similar inequality holds in the other direction. That implies the claim.
\end{proof}

\paragraph{Concentration of key estimates}
The PS estimator defined in~\eqref{def: s-vh} relies on three key estimates,
whose concentration we establish now. 

We begin with 
the concentration of $\estat{k}$ by showing that our estimates of block transition probabilities are concentrated, which boils down to proving that the $\epbtr{uv}$'s are concentrated.
Recall that Claim~\ref{claim:block-transition} implies that the block
transition probabilities are concentrated at each $i$, i.e., the $\btr{i}{uv}$'s are concentrated.  Proving that the estimate 
$\epbtr{uv}
=
\etrans_{u'v}/\etrans_{u'}$ itself is concentrated requires 
an argument. Indeed, as shown in Section~\ref{app:negative} below,
both the numerator and denominator of this estimator in general may have variance asymptotically much greater than $1/\etrans$.
Instead, we use the Markovian structure of the model to control the deviation of $\epbtr{uv}$. 
\begin{claim}[Concentration of block-wise steady-state probability estimates]
	\label{claim:estat-conc}
Conditioned on $\gr$ and $\eventdeg$,
 there exists $c_5 > 0$ such that, for any block $k$, with probability at least $1 - \eps'/2$,
$$
\left|
\frac{\estat{k}}
{\pstat{k}}
-
1
\right|
\leq
c_5
\sqrt{\frac{\log \etrans}{\etrans}}.
$$
\end{claim}
\noindent Recall that $\pstatvec$ was defined in~\eqref{eq:pstatvec}.
\begin{proof}
Throughout this proof, we implicitly condition on $\gr$ and assume that
$\eventdeg$ (from Claim \ref{claim:degrees}) holds. 
We let $\tau_0,\ldots,\tau_{n-1}$ be a topological ordering of the vertices of $\tree$, i.e., an ordering such that: if $\tau_{i}$ is an ancestor of $\tau_j$, then $i < j$. For a fixed $\gr$,
we let $\filter_{0}, \ldots, \filter_{n-1}$ be the corresponding filtration, i.e.,
$$
\filter_j
=
\sigma
\left(
X_{\tau_0},
\ldots,
X_{\tau_j}
\right).
$$
Recall that $\tau'$ is the parent of $\tau \neq \tau_0$.
The proof relies on three sub-claims:
\begin{enumerate}
	\item {\it Deviation of $\etrans_{u'v}$:} For $u,v$ and $j = 1,\ldots, n-1$, let
	$$
	\azuid{j} = \ind\{
	\block{X_{\tau_j'}} = u, \block{X_{\tau_j}} = v
	\},
	$$
	where recall that $\block{i}$ is the block of $i$.
	Note that 
	$$
	\sum_{j=1}^{n-1} \azuid{j} = \etrans_{u'v}.
	$$
	We consider the process
	$$
	\eazu{t}
	=
	\sum_{j=1}^{t} \left\{
	\azuid{j} - \E[\azuid{j}\,|\,\filter_{j-1}]
	\right\}, \qquad t=1,\ldots,n-1,
	$$
	with $\eazu{0} = 0$.
	We claim that $\{\eazu{t}\}_t$ is a martingale with bounded increments. Indeed, by the ordering of the samples, $X_{\tau_j'} \in \filter_j$ since $\tau_j' = \tau_{s}$ for some $s < j$. Hence
	$\azuid{j} - \E[\azuid{j}\,|\,\filter_{j-1}] \in \filter_t$ for all $j \leq t$. So $\eazu{t} \in \filter_t$. 
	Moreover, following a standard calculation,
	$$
	\E[\eazu{t} - \eazu{t-1}\,|\,\filter_{t-1}] 
	= 	
	\E[\azuid{t}\,|\,\filter_{t-1}]  - \E[\E[\azuid{t}\,|\,\filter_{t-1}]\,|\,\filter_{t-1}]
	= 0.
	$$
	Finally, observe that by definition 
	$$
	\left|\eazu{t} - \eazu{t-1}\right|
	= 
	\left|\azuid{t} - \E[\azuid{t}\,|\,\filter_{t-1}]\right|  \leq 1.
	$$
	By the Azuma-Hoeffding inequality (see e.g.~\cite{MotwaniRaghavan:95}), for a constant $c_5 > 0$ large enough
	\begin{eqnarray}
	&&\P\left[
	\left|
	\etrans_{u'v}
	-
	\sum_{j=1}^{n-1} \E[\azuid{j}\,|\,\filter_{j-1}]
	\right|
	\geq 
	c_5 \sqrt{n \log n}
	\right]\nonumber\\
	&&=
	\P
\left[
\left|
\eazu{n-1}
- \eazu0
\right| \geq 
c_5 \sqrt{n \log n}
\right]\nonumber\\
&&\leq
2\exp\left(
- \frac{[c_5\sqrt{n \log n}]^2}{n-1}
\right)\nonumber\\
&&\leq 
\frac{\eps'}{4K^2}.\label{eq:estat-conc-1}
	\end{eqnarray}
	
	\item {\it Deviation of $\sum\limits_{j=1}^{n-1} \E[\azuid{j}\,|\,\filter_{j-1}]$:} Next, we bound 
	\begin{eqnarray}
	&&\sum_{j=1}^{n-1} \E[\azuid{j}\,|\,\filter_{j-1}]\nonumber\\
&&= \sum_{j=1}^{n-1} \P[\block{X_{\tau_j'}} = u, \block{X_{\tau_j}} = v\,|\,\filter_{j-1}]\nonumber\\
&&= \sum_{j=1}^{n-1} \btr{X_{\tau_j'}}{uv},\label{eq:condition-sum-1}
	\end{eqnarray}
	where we use the Markov property of the walk indexed by $\tree$.
	By Claim~\ref{claim:block-transition}, for all $i \in \ve_{u}$,
	\begin{equation}
	\label{eq:condition-sum-2}
	\left|
\frac{\btr{i}{uv}}
{\pbtr{uv}}
-
1
\right|
\leq
c_5'
\sqrt{
	\frac{\log \nve}{\nve}
}.
	\end{equation}
	Combining~\eqref{eq:condition-sum-1} and~\eqref{eq:condition-sum-2},
	we get
	\begin{eqnarray}
	&&\left|
	\sum_{j=1}^{n-1} \E[\azuid{j}\,|\,\filter_{j-1}]
	-
	\pbtr{uv}
	\etrans_{u'} 
	\right|\nonumber\\
	&&\leq 
	c_5'
	\sqrt{
		\frac{\log \nve}{\nve}
	}
\etrans_{u'} \nonumber\\
	&&\leq 
	c_5'
	\sqrt{
		\frac{\log \nve}{\nve}
	}
n \nonumber\\
	&&=
	c_5'
	\sqrt{
		\frac{\log \nve}{\nve}
	}
\sqrt{\frac{n}{\log n}} 
\sqrt{n \log n} \nonumber\\
&&\leq
c_5' \sqrt{n \log n},\label{eq:estat-conc-2}
	\end{eqnarray}
	where we used that $\etrans \leq \nve$ 
    and that $x/\ln x$ is non-decreasing for $x \geq e$.

\item {\it Lower bound on $\etrans_{u'}$:} Let $\nint$ be the number
of internal vertices in $\tree$. Because each leaf has a parent that is an internal vertex and $\tree$ has maximum degree $\dmax\leq c_d$ for some constant $c_d>0$, it follows that $\nint = \Theta(n)$.
Moreover, by Claim~\ref{claim:landing}, the state of each internal vertex of $\tree$ (except the root) has probability
at least $p_*$ of coming from block $u$, independently of all other $X_\tau$'s. As a result, $\etrans_{u'}$ stochastically dominates a binomial random
variable with $\nint - 1$ trials and probability of success $p_*$. By Hoeffding's
inequality we therefore have for a constant $c_6 > 0$ large enough that
\begin{equation*}
\P\left[
\etrans_{u'}
-  p_* (\nint - 1)
<
c_6 \sqrt{n \log n}
\right]
\leq
\exp\left(
- \frac{2  \left[
	c_6 \sqrt{n \log n}
	\right]^2}{\nint -1}
\right)
\leq \frac{\eps'}{4K}.
\end{equation*}
Together with $\nint = \Theta(n)$, that implies that for some constant $c_6' > 0$
\begin{equation}
\label{eq:estat-conc-3}
\P\left[
\etrans_{u'}
\geq
c_6' n
\right]
\geq 1 - \frac{\eps'}{4K}.
\end{equation}

\end{enumerate}

Combining~\eqref{eq:estat-conc-1},~\eqref{eq:estat-conc-2},
and~\eqref{eq:estat-conc-3}, with probability at least 
$1 - \eps'/2$ for any block $u,v$, there exists some constant $c_6'' > 0$, such that 
\begin{eqnarray}
&&\left|
\frac{\epbtr{uv}}{\pbtr{uv} }
- 1
\right|\nonumber\\
&&= 
\left|
\frac{\etrans_{u'v}}{\etrans_{u'}\pbtr{uv} }
- 1
\right|\nonumber\\
&&=
\left|
\frac{\etrans_{u'v}-\etrans_{u'}\pbtr{uv}}{\etrans_{u'}\pbtr{uv} }
\right|\nonumber\\
&&\leq 
\frac{\left|
\etrans_{u'v}-\sum_{j=1}^{n-1} \E[\azuid{j}\,|\,\filter_{j-1}]
\right| +
\left|
\sum_{j=1}^{n-1} \E[\azuid{j}\,|\,\filter_{j-1}]-\etrans_{u'}\pbtr{uv} 
\right|}{\left|
\etrans_{u'}\pbtr{uv}
\right|}\nonumber\\
&&\leq 
|c_6'n\pbtr{uv}|^{-1}(c_5+c_5')\sqrt{n\log n}\nonumber\\
&&\leq 
c_6'' 
\sqrt{\frac{\log n}{n}}. \label{eq:estat-conc-4}
\end{eqnarray}

Recall that the stationary distribution of $P^B$ is
\begin{equation}\label{eq:stationary1}
\pstat{k} = \left[\sum\limits_v\frac{\pbtr{kv}}{\pbtr{vk}}\right]^{-1},
\end{equation}
for any $k\in\{1,\dots,K\}$ and that
\begin{equation}\label{eq:stationary2}
\estat{k} 
=  \left[\sum\limits_v\frac{\epbtr{kv}}{\epbtr{vk}}\right]^{-1}.
\end{equation}
Then, there exists some constant $c_6''' > 0$, such that 
$$\left| 
\frac{\estat{k}}{\pstat{k}} 
- 
1 
\right|
\leq
c_6'''
\sqrt{\frac{\log n}{n}}.
$$
Indeed,
\begin{align*}
&
\frac{\pstat{k}}{\estat{k}}
- 
1\\ 
&= 
\frac{\sum\limits_v \epbtr{kv}/\epbtr{vk}}
{\sum\limits_v \pbtr{kv}/\pbtr{vk}}
- 
1\\
&= 
\frac{\sum\limits_v (\pbtr{kv}/\pbtr{vk})(\epbtr{kv}/\pbtr{kv})(\pbtr{vk}/\epbtr{vk})}
{\sum\limits_v \pbtr{kv}/\pbtr{vk}}
- 
1\\
&\leq  \frac{1+c_6''\sqrt{\frac{\log n}{n}}}{1-c_6''\sqrt{\frac{\log n}{n}}}-1\\
&\leq c_6'''\sqrt{\frac{\log n}{n}},
\end{align*}
for large enough $c_6'''$, and similarly in the other direction.  The second line is from \eqref{eq:stationary1} and \eqref{eq:stationary2} while fourth line is from \eqref{eq:estat-conc-4}.
\end{proof}

We then evaluate
the deviation of $$
\emfunc{k} = \frac{\edeg{k}}{n_k}\sum\limits_{\tau\in\mathbbm{T}_k}\frac{Y_{\tau}}{\deg{X_{\tau}}}.
$$
Recall that, for any block $k$,
the population block-wise average is
$$
\mfunc_k
=
\frac{1}{\bsize{k}} \sum_{i\in \ve_k} y_i.
$$  
Before showing that our block-wise estimator $\emfunc{k}$ is close to $\mfunc_k$, we first look at a related quantity, $\wfunc{k}$ below, which serves as a ``bridge.''  
We define the weighted block-wise average as
\begin{equation}
\label{eq:wfunc-def}
\wfunc{k}
=
\frac{1}{n_k}\sum\limits_{\tau\in\mathbbm{T}_k} \frac{Y_{\tau}}{N_k\theta_{X_{\tau}}}.
\end{equation} 
Using an argument similar to that in Claim \ref{claim:estat-conc}, 
we show in Claim \ref{claim:emfunc-conc} that $\wfunc{k}$ is concentrated for each block $k$.  We then show in Claim \ref{claim:close} that $\wfunc{k}$ is close to $\emfunc{k}$.  
As a result, we will have established that $\emfunc{k}$ is close to $\mfunc_k$.

\begin{center}
\begin{tikzpicture}[
 every matrix/.style={ampersand replacement=\&,column sep=3cm,row sep=0.2cm},
c/.style={draw,thick,rounded corners,fill=yellow!20,inner sep=.2cm},
to/.style={->,>=stealth',shorten >=1pt,thick},
every node/.style={align=center}]

  \matrix{
      \node[c] (c1) { $\emfunc{k}$
       };
      \& \node[c] (c2) { $\wfunc{k}$
   }; 
      \&  \node[c] (c3) { $\mfunc_k$ 
       };  \\
  };

   \draw[to] (c1) -- node[midway,above] {close to}
  node[midway,below] {Claim \ref{claim:close}} (c2);
   \draw[to] (c2) -- node[midway,above] {close to}
node[midway,below] {Claim \ref{claim:emfunc-conc}} (c3);

 \end{tikzpicture}
\end{center}

\begin{claim}[Concentration of block-wise sample averages weighted by degrees]
	\label{claim:emfunc-conc}
	Conditioned on $\gr$, $\eventdeg$ and $\eventdegtwo$,
	there exists $c_7 > 0$ such that, with probability at least $1 - \eps'/4$,
	for any block $k$
	$$
	\left|
	\frac{\wfunc{k}}
	{\mfunc_k}
	-
	1
	\right|
	\leq
	c_7
	\sqrt{\frac{\log n}{n}}.
	$$
\end{claim}
\begin{proof}
Because the structure of the proof is similar to that of Claim~\ref{claim:estat-conc}, we only sketch it here.
We also make use of Claim~\ref{claim:two-step-mixing}, which shows that simple random walk on $\gr$ mixes well {\it within blocks} in two steps. Because of the latter,
we control separately the odd and even levels of $\tree$. 
Let $\nu_1,\nu_2,\ldots,\nu_{\neven}$ be the vertices of $\tree$ whose graph distance
to the root is even, including the root $\nu_1 = \tau_0$, in a topological ordering. Let $\childtwo(\nu)$ be the grand-children of $\nu$ in $\tree$. Let $\filtertwo_0
= \sigma(X_{\nu_1}) = \sigma(X_{\tau_0})$ and for $j \geq 1$
$$
\filtertwo_j
=
\filtertwo_0 \cup 
\sigma(X_{\nu}\,:\,\nu \in \childtwo(\nu_{\ell}), \ell \leq j).
$$
For each node $X_\nu\in V_k$, define 
$$\func_{\theta}(X_{\nu}) = \frac{\func(X_{\nu})}{\nve_{Z_\nu}\theta_{X_\nu}}.$$
Fix block $k$ and let 
$$
\azuid{j} 
= 
\sum_{\nu \in \childtwo(\nu_j)}
\ind\{
\block{X_{\nu}} = k
\} \,\func_{\theta}(X_{\nu}),
$$
and
$$
\etranseven_{k}
:=
\sum_{i=2}^{\neven}  \ind\{\block{X_{\nu_i}} = k \},
$$
where note that the last sum excludes the root.
Following the proof of Claim~\ref{claim:estat-conc},
we note that the partial sums
$$
\sum_{j=1}^{J} \left\{
\azuid{j} - \E[\azuid{j}\,|\,\filtertwo_{j-1}]
\right\}, \qquad J=1,\ldots,\neven,
$$
form a martingale indexed by $J$ with increments satisfying
$$
\left|
\azuid{j} - \E[\azuid{j}\,|\,\filtertwo_{j-1}]
\right|
\leq (c_d-1)^2 \bfunc,
$$
where we used that $\tree$ has maximum degree $\leq c_d$ 
and $0 \leq \func(x) \leq \bfunc$ by assumption. Hence, arguing as 
in Step 1 of Claim~\ref{claim:estat-conc}, we get that
with probability at least $1 - \eps'/20K$ for all $k$
\begin{equation}
\label{eq:emfunc-conc-1}
\left|
\sum_{i=2}^{\neven}  \ind\{\block{X_{\nu_i}} = k \} \,\func_{\theta}(X_{\nu_i})
-
\sum_{j=1}^{\neven} \E[\azuid{j}\,|\,\filtertwo_{j-1}]
\right|
\leq 
c_7' \sqrt{n \log n}.
\end{equation}
Moreover, let $\nu \in \childtwo(\nu_j)$ and notice that by construction
$X_{\nu_j} \in \mathcal{G}_{j-1}$. Hence, by Claim~\ref{claim:two-step-mixing}, 
\begin{eqnarray*}
&&\E[\ind\{
\block{X_{\nu}} = k
\} \,\func_{\theta}(X_{\nu})\,|\,\filtertwo_{j-1}]\\
&&\leq
\sum_{x \in \ve_k} 
\left\{
\theta_{x} \sum_{u} \pbtr{\block{X_{\nu_j}},u} \,\pbtr{uk}
\left[1 + c_4 \sqrt{\frac{\log \nve}{\nve}}\right] 
\right\}\func_{\theta}(x)\\
&&\leq
\sum_{x \in \ve_k} 
\left\{
\nve_k^{-1}\sum_{u} \pbtr{\block{X_{\nu_j}},u} \,\pbtr{uk}
\left[1 + c_4 \sqrt{\frac{\log \nve}{\nve}}\right] 
\right\}\func(x)\\
&&= \mfunc_k \sum_{u} \pbtr{\block{X_{\nu_j}},u} \,\pbtr{uk}
\left[1 + c_4 \sqrt{\frac{\log \nve}{\nve}}\right] ,
\end{eqnarray*}
where recall that we condition on $G$.
  Similarly in the opposite direction. So,
arguing as in Step 2 of Claim~\ref{claim:estat-conc}, for some large enough $c_7''>0$, 
\begin{eqnarray}
&&\left|
\sum_{j=1}^{\neven} \E[\azuid{j}\,|\,\filtertwo_{j-1}]
-
\mfunc_k
\sum_{v }
\sum_{u} \pbtr{vu} \pbtr{uk}
\sum_{j=1}^{\neven}
\ind\{\block{X_{\nu_j}} = v \}
\left|
\childtwo(\nu_j)
\right|
\right|\nonumber\\
&&\leq c_4\sqrt{\frac{\log \nve}{\nve}}n c_d^2 c_y\nonumber\\
&&\leq c_4 c_d^2 c_y\sqrt{\frac{\log \nve}{\nve}}\sqrt{\frac{n}{\log n}}\sqrt{n\log n}\nonumber\\
&& \leq
c_7'' \sqrt{n \log n},\label{eq:emfunc-conc-2}
\end{eqnarray}
where we used $n \leq N$.

In addition, we argue as in Step 3 of Claim~\ref{claim:estat-conc}. 
Because each node with odd distance to the root has a parent with even distance to the root, and $\tree$ has maximum degree $\leq c_d$, 
it follows that $n^{(e)} = \Theta(n)$.
Moreover, by Claim~\ref{claim:landing}, the state of each internal vertex of $\tree$ (except the root) has probability
at least $p_*$ of coming from block $u$, independently of all other $X_\tau$'s. As a result, $\etranseven_{k}$ stochastically dominates a binomial random
variable with $n^{(e)} - 1$ trials and probability of success $p_*$. By Hoeffding's
inequality we therefore have for a constant $c_8 > 0$ large enough that
\begin{equation*}
\P\left[
\etranseven_{k}
-  p_* (n^{(e)} - 1)
<
c_8 \sqrt{n \log n}
\right]
\leq
\exp\left(
- \frac{2  \left[
	c_8 \sqrt{n \log n}
	\right]^2}{n^{(e)} -1}
\right)
\leq \frac{\eps'}{20K}.
\end{equation*}
Together with $n^{(e)} = \Theta(n)$, that implies that with probability at least $1 - \eps'/10$ for all block $k$ for some constant $c_8' > 0$
\begin{equation}
\label{eq:emfunc-conc-3}
\etranseven_{k}
\geq
c_8' n.
\end{equation}

Finally, following the proof of Claim~\ref{claim:estat-conc} once again,
we also get that with probability at least $1 - \eps'/10$ for all $k, $ 
\begin{eqnarray}
&&\Bigg|
\sum_{u} \pbtr{vu} \pbtr{uk}
\sum_{j=1}^{\neven}
\ind\{\block{X_{\nu_j}} = v \}
\left|
\childtwo(\nu_j)
\right|\nonumber\\
&& \qquad -
\sum_{j=1}^{\neven}
\sum_{\nu \in \childtwo(\nu_j)}
\ind\{\block{X_{\nu_j}} = v,
\block{X_{\nu}} = k\}
\Bigg|
\leq 
c_8'' \sqrt{n \log n},
\label{eq:emfunc-conc-4}
\end{eqnarray}
for some constant $c_8'' > 0$.
Combining~\eqref{eq:emfunc-conc-1},~\eqref{eq:emfunc-conc-2}, and~\eqref{eq:emfunc-conc-4}, with probability at least $1-\eps'/5$ 
\begin{eqnarray}
&&\left|
\sum_{i=2}^{\neven}  \ind\{\block{X_{\nu_i}} = k \} \,\func_{\theta}(X_{\nu_i})
-
\mfunc_k
\sum_{i=2}^{\neven}  \ind\{\block{X_{\nu_i}} = k \}
\right|\nonumber\\
&&\leq  
\left|
\sum_{i=2}^{\neven}  \ind\{\block{X_{\nu_i}} = k \} \,\func_{\theta}(X_{\nu_i})
-
\sum_{j=1}^{\neven} \E[\azuid{j}\,|\,\filtertwo_{j-1}]
\right|\nonumber\\
&&\quad+
\left|
\sum_{j=1}^{\neven} \E[\azuid{j}\,|\,\filtertwo_{j-1}]
-
\mfunc_k
\sum_{v }
\sum_{u} \pbtr{vu} \pbtr{uk}
\sum_{j=1}^{\neven}
\ind\{\block{X_{\nu_j}} = v \}
\left|
\childtwo(\nu_j)
\right|
\right|\nonumber\\
&&\quad+
\Bigg|\mu_k\sum\limits_v
\sum_{u} \pbtr{vu} \pbtr{uk}
\sum_{j=1}^{\neven}
\ind\{\block{X_{\nu_j}} = v \}
\left|
\childtwo(\nu_j)
\right| \nonumber\\
&&\quad-
\mu_k\sum\limits_v\sum_{j=1}^{\neven}
\sum_{\nu \in \childtwo(\nu_j)}
\ind\{\block{X_{\nu_j}} = v,
\block{X_{\nu}} = k\}
\Bigg|\nonumber\\
&&\leq  
c_8'''\sqrt{n \log n},\nonumber
\end{eqnarray}
for some constant $c_8''' > 0$.

The same holds for the odd levels. Together with~\eqref{eq:emfunc-conc-3} and a similar inequality for odd levels (and the fact that the first two levels of $\tree$ have negligible effect asymptotically), 
we get the claim.
\end{proof}	

By replacing $y(X_{\tau})$ by 1 in the proof of Claim \ref{claim:emfunc-conc}, we can also derive the following.
\begin{claim}
	\label{claim: help_degree}
	Conditioned on $\gr$, $\eventdeg$ and $\eventdegtwo$,
	there exists $c_9 > 0$ such that, with probability at least $1 - \eps'/4$,
	for any block $k$
	$$
	\left|
	\frac{1}{n_k}\sum\limits_{\tau\in\mathbbm{T}_k} \frac{1}{N_k\theta_{X_{\tau}}}
	-
	1
	\right|
	\leq
	c_9
	\sqrt{\frac{\log n}{n}}.
	$$
\end{claim}

Using Claims \ref{claim:degrees} and \ref{claim: help_degree}, we derive the deviation of the block-wise harmonic average degrees.
Recall, for any block $k$, the block population mean degree is
$$
\mdegb{k} 
= \frac{B_{k\ast}}{N_k},
$$
and the block-wise harmonic average degree as 
\begin{equation*}
\edeg{k} = \left(\frac{1}{n_k}\sum\limits_{\tau\in\mathbbm{T}_k}\frac{1}{d_{X_{\tau}}}\right)^{-1}.
\end{equation*}
\begin{claim}[Concentration of block-wise harmonic average of degrees]
	\label{claim:degrees-conc}
	Conditioned on $\gr$, $\eventdeg$ and $\eventdegtwo$, 
	there exists $c_{10} > 0$ such that, with probability at least $1 - \eps'/4$,
	for any block $k$,
	$$
	\left|
	\frac{\edeg{k}}{\mdegb{k}}
	-
	1
	\right|
	\leq
	c_{10}
	\sqrt{\frac{\log \etrans}{\etrans}}.
	$$
\end{claim}
\begin{proof}
Conditioned on $\gr$, $\eventdeg$ and $\eventdegtwo$, under the DC-SBM, 
	\begin{align*}
	(\edeg{k})^{-1}
	 &= n_k^{-1}\sum\limits_{\tau\in\mathbbm{T}_k}1/d_{X_{\tau}}\\
	 &\leq n_k^{-1}\sum\limits_{\tau\in\mathbbm{T}_k}\left[\theta_{X_{\tau}}B_{k\ast}\left(1-Kc_1\sqrt{\frac{\log N}{N}}\right)\right]^{-1}\\
	 &= \left(1-Kc_1\sqrt{\frac{\log N}{N}}\right)^{-1}n_k^{-1}\sum\limits_{\tau\in\mathbbm{T}_k}\frac{1}{\theta_{X_{\tau}}B_{k\ast}}\\
	 &= \left(1-Kc_1\sqrt{\frac{\log N}{N}}\right)^{-1}\left[n_k^{-1}\sum\limits_{\tau\in\mathbbm{T}_k}\frac{1}{N_k\theta_{X_{\tau}}}\right]\frac{N_k}{B_{k\ast}}\\
	 &\leq
	 \left(1-Kc_1\sqrt{\frac{\log N}{N}}\right)^{-1}\left(1+c_9\sqrt{\frac{\log n}{n}}\right)\frac{N_k}{B_{k\ast}}\\
	 &\leq 
	 \left(1+c_{10}'\sqrt{\frac{\log n}{n}}\right)\frac{N_k}{B_{k\ast}}
     ,
	\end{align*}
for some large enough constant $c_{10}'>0$.  The first inequality is from Claim \ref{claim:degrees}, which holds with probability $1- \eps'/4$.  The second inequality is from Claim \ref{claim: help_degree}.  
A similar inequality holds for the opposite direction.
	Thus, $$
	\left|
	\frac{\nve_k^{-1}\aff{k\ast}}{\edeg{k}}
	-
	1
	\right|
	\leq
	c_{10}'
	\sqrt{\frac{\log \etrans}{\etrans}}.
	$$
	Thus, 
	\begin{eqnarray*}
    \label{eq:deg_concen_1}
	\left|
	\frac{\edeg{k}}{\nve_k^{-1}\aff{k\ast}}
	-
	1
	\right|
	\leq
	c_{10}
	\sqrt{\frac{\log \etrans}{\etrans}},
	\end{eqnarray*}
	for some large enough constant $c_{10}>0$.
	By definition of $\mdegb{k}$ we are done.
%
%
\end{proof}

Directly from Claim \ref{claim:degrees-conc}, we show that $\wfunc{k}$ is close to $\emfunc{k}$ for each block $k$ in the following claim. 

\begin{claim}
	\label{claim:close}
	Conditioned on $\gr$, $\eventdeg$ and $\eventdegtwo$,
	there exists $c_{11} > 0$ such that, with probability at least $1 - \eps'/2$,
	for any block $k$,
	$$
	\left|
	\frac{\emfunc{k}}{\mfunc_k}
	-
	1
	\right|
	\leq
	c_{11}
	\sqrt{\frac{\log \etrans}{\etrans}}.
	$$
	
\end{claim}

\begin{proof}
	Conditioned on $\gr$, $\eventdeg$ and $\eventdegtwo$, under the DC-SBM, Claims \ref{claim:degrees-conc} and \ref{claim:emfunc-conc}
	hold simultaneously with probability $1- \eps'/2$. Then	
	\begin{align*}
	\emfunc{k} 
	= &n_k^{-1}\sum\limits_{\tau\in\mathbbm{T}_k}\frac{Y_{\tau}\edeg{k}}{\deg{X_{\tau}}}\\
	=
	&\frac{\edeg{k}}{\mdegb{k}}\, n_k^{-1}\sum\limits_{\tau\in\mathbbm{T}_k}\frac{Y_{\tau}B_{k\ast}}{\deg{X_{\tau}}\nve_k}\\
	\leq
	&\left(1+c_{10}\sqrt{\frac{\log n}{n}}\right) n_k^{-1}\sum\limits_{\tau\in\mathbbm{T}_k}\frac{Y_{\tau}B_{k\ast}}{\deg{X_{\tau}}\nve_k}\\
	\leq
	&\left(1+c_{10}\sqrt{\frac{\log n}{n}}\right) n_k^{-1}\sum\limits_{\tau\in\mathbbm{T}_k}\frac{Y_{\tau}B_{k\ast}}{\theta_{X_{\tau}}B_{k\ast}\left(1-c_1\sqrt{\frac{\log\nve}{\nve}}\right)\nve_k}\\
	=
	&\left(1+c_{10}\sqrt{\frac{\log n}{n}}\right)\left(1-c_1\sqrt{\frac{\log\nve}{\nve}}\right)^{-1} \left[n_k^{-1}\sum\limits_{\tau\in\mathbbm{T}_k}\frac{Y_{\tau}}{\theta_{X_{\tau}}\nve_k}\right]\\
	=
	&\left(1+c_{10}\sqrt{\frac{\log n}{n}}\right)\left(1-c_1\sqrt{\frac{\log\nve}{\nve}}\right)^{-1}\wfunc{k}\\
	\leq
	&\left(1+c_{11}'\sqrt{\frac{\log n}{n}}\right)\wfunc{k},
	\end{align*}
	for some large enough constant $c_{11}'>0$.  
	The first inequality is from Claim \ref{claim:degrees-conc}, while
	the second inequality is from Claim \ref{claim:degrees}.  A similar bound holds for the opposite direction.  Combining with Claim \ref{claim:emfunc-conc}, 
	$$
	\left|
	\frac{\emfunc{k}}{\mfunc_k}
	-
	1
	\right|
	\leq
	c_{11}
	\sqrt{\frac{\log \etrans}{\etrans}},
	$$
	for some large enough constant $c_{11} >0$.
\end{proof}

\paragraph{Putting everything together} 
Finally, we prove the main result.
\begin{proof}[Proof of Theorem~\ref{thm: consistency}]
By Claims~\ref{claim:degrees} and~\ref{claim:degreetwos},
events $\eventdeg$ and $\eventdegtwo$ hold with probability
at least $1 - \eps$. Under those events, by Claims~\ref{claim:estat-conc} and ~\ref{claim:degrees-conc} with hold with probability $1-\eps'$,
 \begin{align*}
 \frac{\hat{\pi}_k^{B}}{\edeg{k}}
 \leq
 \frac{\pstat{k}\left(1+c_5\sqrt{\frac{\log n}{n}}\right)}{\mdegb{k}\left(1-c_{10}\sqrt{\frac{\log n}{n}}\right)}
 =
 \frac{\pstat{k}}{\mdegb{k}}\left(1+c_{12}'\sqrt{\frac{\log n}{n}}\right),
 \end{align*}
for some large enough $c_{12}'>0$.  Similar for the other direction.  
Then, using Claim ~\ref{claim:close}, 
\begin{eqnarray*}
\ps&=& \frac{\sum_k [\hat{\pi}_k^{B}/\edeg{k}]\,\emfunc{k}}{\sum_k\hat{\pi}_k^{B}/\edeg{k}}
\nonumber\\
&\leq&
\frac{\sum\limits_k\left[ \pstat{k}/\mdegb{k}\right]\left(1+c_{12}'\sqrt{\frac{\log n}{n}}\right)
\,\mu_k\,\left(1+c_7\sqrt{\frac{\log n}{n}}\right)}{\sum\limits_k \left[\pstat{k}/\mdegb{k}\right]\left(1-c_{12}'\sqrt{\frac{\log n}{n}}\right)}\nonumber\\
&\leq&
 \frac{\sum\limits_k\left[ \pstat{k}/\mdegb{k}\right]\mu_k}{\sum\limits_k \left[\pstat{k}/\mdegb{k}\right]}\left(1+c_{12}\sqrt{\frac{\log n}{n}}\right)\nonumber\\
&=& \mu_{\mathrm{true}}\left(1+c_{12}\sqrt{\frac{\log n}{n}}\right),\nonumber
\end{eqnarray*}
for some constant $c_{12} > 0$.  Similarly for the other direction.  Thus, there exists constant $c>0$ such that
$$
\left|
\ps -\mu_{\mathrm{true}}
\right|
\leq
c \sqrt{\frac{\log n}{n}}.
$$
\end{proof}

\subsection{A simple instance showing that the variance of the VH estimator converges slower than $O(n^{-1})$}
\label{app:negative}
The following example shows that, in general,
the Volz-Heckathorn estimator, i.e., 
$$
\vh
=
\frac{
\sum_{\tau \in \tree} \func(X_{\tau})/\deg{X_{\tau}}
}
{
\sum_{\tau \in \tree} 1/\deg{X_{\tau}}
},
$$
has a variance asymptotically
worse than $1/n$ on a two-block stochastic block model. Recall that $z(x)$ is the block of $x$.
\begin{theorem}[Negative example]
Let $K=2$ and denote the blocks by $\{0,1\}$. Let $\bsize{0} = \bsize{1} = \nve/2$, $\aff{01} = \aff{10} = 1 - \aff{00} = 1 - \aff{11} = p N^2$ where $p \in (0,1/2)$, $\func(x) = \block{x}$ for all $x \in \ve$. Let $x_0 \in \ve$
be chosen uniformly at random. Let $\tree$ be a complete $(\alpha-1)$-ary tree. Assume that $\nve \gg n^{2+\gamma}$ for some $\gamma > 0$ and
that 
\begin{equation}
\label{eq:ks}
2(1-2p)^2 > 1.
\end{equation}
Then, with probability at least $1/2$ over the network,  
$$
\var\left[
\vh
\,\middle|\,\gr
\right]
\gg
\frac{1}{n^{1-\zeta}},
$$
for some $\zeta > 0$.
\end{theorem}
\begin{proof}
By Claim~\ref{claim:degrees},
the event $\eventdeg$ occurs with probability at least
$1/2$. Therefore, by the conditional variance formula,
\begin{equation}
\label{eq:neg-1}
\var\left[
\vh
\,\middle|\,\gr
\right]
\geq
\frac{1}{2}
\var\left[
\vh
\,\middle|\,\gr,\eventdeg
\right].
\end{equation}
By symmetry, $\mdegb{0} = \mdegb{1} = \nve/2$. 	
Hence, on $\eventdeg$, we have further that 
\begin{eqnarray}
&&\var\left[
\vh
\,\middle|\,\gr,\eventdeg
\right]\nonumber\\
&&=
\E\left[
\left(
\frac{
	\sum_{\tau \in \tree} \func(X_{\tau})/\deg{X_{\tau}}
}
{
	\sum_{\tau \in \tree} 1/\deg{X_{\tau}}
}
\right)^2
\,\middle|\,\gr,\eventdeg
\right]\nonumber\\
&&\quad -
\left(
\E\left[
\frac{
	\sum_{\tau \in \tree} \func(X_{\tau})/\deg{X_{\tau}}
}
{
	\sum_{\tau \in \tree} 1/\deg{X_{\tau}}
}
\,\middle|\,\gr,\eventdeg
\right]
\right)^2
\nonumber\\
&&\geq
\E\left[
\left(
\frac{
	\sum_{\tau \in \tree} \func(X_{\tau})
}
{
	n
}
\right)^2
\,\middle|\,\gr,\eventdeg
\right]\nonumber\\
&&\quad -
\left(
\E\left[
\frac{
	\sum_{\tau \in \tree} \func(X_{\tau})
}
{
	n
}
\,\middle|\,\gr,\eventdeg
\right]
\right)^2
- O\left(
\sqrt{\frac{\log \nve}{\nve}}
\right)
\nonumber\\
&&=
\var\left[
\frac{1}{n} \sum_{\tau \in \tree} \func(X_{\tau})
\,\middle|\,\gr,\eventdeg
\right]
- 
o(n^{-1}),\label{eq:neg-2}
\end{eqnarray}
by our assumption on $\nve$, where we used that $\func(x) \in [0,1]$ for all $x$.
To simplify notation, in the rest of the proof, we implicitly condition
on $\gr$ and $\eventdeg$. 

The population-level chain satisfies
$$
\pbtr{00}
=
\pbtr{11}
= 1-p,
\qquad
\pbtr{01}
=
\pbtr{10}
= p,
\qquad
\pstat{0}
=
\pstat{1}
=
\frac{1}{2}.
$$
Let $(\tilde{f}_\tau)_{\tau \in \tree}$ be a Markov chain on $\{0,1\}$ indexed by $\tree$ with
transition probabilities $(\pbtr{bu})_{bu\in\{0,1\}}$. By Claim~\ref{claim:block-transition}, on $\eventdeg$, we can couple 
$(\func(X_{\tau}))_\tau$ and $(\tilde{f}_\tau)_\tau$
except with probability $O(n \sqrt{\log \nve/\nve}) = o(1)$, an event we denote by $\tilde{\mathcal{E}}$.  This is because, for each of the $n-1$ transitions, there can only be a difference in probability of $O(\sqrt{\log \nve/\nve})$.
Hence, by the conditional variance formula again,
\begin{eqnarray}
&&\var\left[
\frac{1}{n} \sum_{\tau \in \tree} \func(X_{\tau})
\right]\nonumber\\
&&\geq (1-o(1)) \var\left[
\frac{1}{n} \sum_{\tau \in \tree} \func(X_{\tau})
\,\middle|\,\tilde{\mathcal{E}}\right]\nonumber\\
&&= (1-o(1)) \var\left[
\frac{1}{n} \sum_{\tau \in \tree} \tilde{f}_{\tau}
\,\middle|\,\tilde{\mathcal{E}}\right].\label{eq:neg-3}
\end{eqnarray}
To simplify notation, in the rest of the proof, we implicitly condition
on $\tilde{\mathcal{E}}$. 

Define 
$$
\tilde{g}_\tau
:=
1 - 2 \tilde{f}_\tau
\in \{-1,+1\},
$$
and notice that, by translation,
\begin{equation}
\label{eq:neg-4}
\var\left[
\frac{1}{n} \sum_{\tau \in \tree} \tilde{f}_{\tau}
\right]
=
\frac{1}{4}
\var\left[
\frac{1}{n} \sum_{\tau \in \tree} \tilde{g}_{\tau}
\right]
\end{equation}
and that $\tilde{g}_\tau$ is centered under $\pstatvec$.
Under $(\pbtr{bu})_{bu\in\{0,1\}}$, the function $(-1,+1)$ is a right-eigenvector
with eigenvalue 
$$
\theta := 1-2p \in (0,1).
$$
Hence, for any $\tau, \tau' \in \tree$ at graph distance $\eta$,
it holds that
$$
\E[\tilde{g}_{\tau'}\,|\,\tilde{g}_{\tau}] 
=
\theta^{\eta} \tilde{g}_{\tau},
$$
and
$$
\cov[\tilde{g}_{\tau'},\tilde{g}_{\tau}]
=
\E[
\tilde{g}_{\tau'}\tilde{g}_{\tau}
]
=
\E[
\E[
\tilde{g}_{\tau'}\tilde{g}_{\tau}
\,|\,\tilde{g}_{\tau}
]
]
=
\E[
\theta^{\eta} \tilde{g}_{\tau}^2
]
=
\theta^{\eta},
$$
where we used that $\tilde{g}_{\tau}^2 = 1$.
Let $\mathcal{L}$ be the leaves of $\tree$. Because the
samples $(\tilde{g}_\tau)_{\tau \in \tree}$ are positively correlated
by the above calculation
and $|\mathcal{L}| = \Omega(n)$, we have further that
\begin{equation}
\label{eq:neg-5}
\var\left[
\frac{1}{n} \sum_{\tau \in \tree} \tilde{g}_{\tau}
\right]
=
\Omega\left(
\var\left[
\frac{1}{|\mathcal{L}|} \sum_{\tau \in \mathcal{L}} \tilde{g}_{\tau}
\right]
\right).
\end{equation}

Finally, by symmetry and the conditional variance formula once more,
recalling that $\tau_0$ is the root of $\tree$ we have 
\begin{eqnarray}
\var\left[
\frac{1}{|\mathcal{L}|} \sum_{\tau \in \mathcal{L}} \tilde{g}_{\tau}
\right]
&\geq&
\var\left[
\E\left[
\frac{1}{|\mathcal{L}|} \sum_{\tau \in \mathcal{L}} \tilde{g}_{\tau}
\,\middle|\,
\tilde{g}_{\tau_0}
\right]
\right]\nonumber\\
&=&
\var\left[
\theta^{\log_2 n}
\tilde{g}_{\tau_0}
\right]\nonumber\\
&=& \theta^{2 \log_2 n}\nonumber\\
&=& n^{\log_2 \theta^2}\nonumber\\
&=& \frac{1}{n^{1-\zeta}},
\end{eqnarray}
with $\zeta = 1+\log_2 \theta^2 = \log_2 (2\theta^2) > 0$ by~\eqref{eq:ks}.
Combining the latter with \eqref{eq:neg-1}, \eqref{eq:neg-2}, \eqref{eq:neg-3}, \eqref{eq:neg-4}, and \eqref{eq:neg-5} gives the result. 
\end{proof}

\end{document}